\documentclass[12pt]{amsart}

\usepackage{enumerate}
\usepackage{amsfonts}
\usepackage{amsthm}
\usepackage{amsmath}
\usepackage{amssymb}
\usepackage{amsrefs}
\usepackage{t1enc}
\usepackage{fullpage}
\usepackage{dsfont}
\usepackage{graphicx}
\usepackage{mathrsfs}
\usepackage{mathptmx}
\usepackage[a4paper,left=2cm,right=2cm, footskip=30pt]{geometry}
\usepackage{lipsum}
\usepackage{geometry}
\usepackage[utf8]{inputenc}

\newtheorem{theorem}{Theorem}[section]
\newtheorem{prop}[theorem]{Proposition}
\newtheorem{lemma}[theorem]{Lemma}
\newtheorem{defi}[theorem]{Definition}

\newtheorem{kov}[theorem]{Corollary}
\theoremstyle{definition}
\newtheorem{megj}[theorem]{Remark}
\newtheorem{quest}[theorem]{Question}

\newcommand{\inte}{\mathop{\mathrm{int}}}
\newcommand{\cone}{\mathop{\mathrm{cone}}}
\newcommand{\conv}{\mathop{\mathrm{conv}}}

\begin{document}

\author{Bal\'azs Maga}

\address{E\"otv\"os Lor\'and University, P\'azm\'any P\'eter s\'et\'any 1/C, Budapest, H-1117 Hungary}
\email{magab@cs.elte.hu}

\title{Baire Categorical Aspects of First Passage Percolation II.}

\thanks{\includegraphics[height=5.0mm]{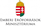}\ Supported by the \'UNKP-18-2 new national excellence program of the Ministry of Human Capacities. Supported by the Hungarian National Research, Development and Innovation Office--NKFIH, Grant 124003.}


\begin{abstract}
In this paper we continue our earlier work about topological first passage percolation and answer certain questions asked in our previous paper. Notably, we prove that apart from trivialities, in the generic configuration there exists exactly one geodesic ray, in the sense that we consider two geodesic rays distinct if they only share a finite number of edges. Moreover, we show that in the generic configuration any not too small and not too large convex set arises as the limit of a sequence $\frac{B(t_n)}{t_n}$ for some $(t_n)\to\infty$. Finally, we define topological Hilbert first passage percolation, and amongst others we prove that certain geometric properties of the percolation in the generic configuration guarantee that we consider a setting linearly isomorphic to the ordinary topological first passage percolation.
\end{abstract}

\maketitle

\section{Introduction}

First passage percolation was introduced by Hammersley and Welsh in 1965 (see \cite{HW}) as a model to describe fluid flows through porous medium. It quickly became a popular area of probability theory, as one can easily ask very difficult questions. Many of these have still remained unsolved despite the growing interest from mathematicians, physicists and biologists.

The main setup is the following: we have a given graph, usually we like to consider the lattice $\mathbb{Z}^d$. We denote the set of nearest neighbor edges by $E$. We place independent, identically distributed, non-negative random variables with a distribution law $\mu$ on each edge $e\in{E}$, which is called the passage time of $e$, and denoted by $\tau(e)$. We think about it as the time needed to traverse $e$. Based on this, we can define the passage time of any finite path $\Gamma$ of consecutive edges as the sum of the passage times of contained edges:
\begin{displaymath}
\tau(\Gamma)=\sum_{e\in\Gamma}\tau(e).
\end{displaymath}
\noindent Using this definition, we might define the passage time between any two points, or in other words the $T$-distance of any two points $x,y\in\mathbb{R}^d$
\begin{displaymath}
T(x,y)=\inf_{\Gamma}\tau(\Gamma),
\end{displaymath}
\noindent where the infimum is taken over all the paths connecting $x'$ to $y'$, where $x'$ and $y'$ are the unique lattice points such that $x\in x'+[0,1)^d$, $y\in y'+[0,1)^d$. The term "distance" is appropriate here: one can easily show that $T:\mathbb{Z}^d\times\mathbb{Z}^d\to\mathbb{R}$ is a pseudometric, that is an "almost metric" in which the distance of distinct points might be 0.

In brief, this is the probabilistic setup. In the sequel when we recall results related to this theory, for the sake of brevity we will often omit the precise technical conditions, such as conditions about the finiteness of certain moments or the value of the distribution function in the infimum of its support.  Instead of it we will simply refer to "some mild conditions" about the distribution function and cite the source of the result. For the reader interested in the details the recently published book \cite{FPP} is also warmly recommended.

The topological variant of this theory is defined similarly and appeared in an earlier paper of the author (see \cite{M}). We fix $A\subseteq[0,\infty)$. To exclude trivialities, let $A$ have at least two elements. The passage time $\tau(e)$ of any edge $e$ will be an element of $A$, and passage times of paths and between points are defined as in the probabilistic setup: more explicitly, the passage time $\tau(\Gamma)$ of a path $\Gamma$ is simply the sum of passage times of the edges along the path, while for $x,y\in\mathbb{R}^d$ we define the passage time between them by $T(x,y)=\inf_{\Gamma:x'\to y'}\tau(\Gamma)$ where $x\in x'+[0,1)^d, y\in y'+[0,1)^d$ for the lattice points $x',y'$. Formally, the space of configurations is $\Omega=\times_{e\in{E}}A$, equipped with the product topology, while $A$ is considered with its usual subspace topology inherited from $\mathbb{R}$. If there might be ambiguity, we will write $T_\omega$ and $\tau_\omega$ for the passage times in the configuration $\omega\in\Omega$.

Our primary interests are classical questions of the probabilistic setup which make sense in the topological setup as well. More precisely, we examine whether a property which has probability 1 in the probabilistic setup holds in a residual subset of $\Omega$ in the topological setup. In \cite{M} we made some progress in certain problems of this type, however, in some cases we could not give complete answers. This lack of completeness was one of the inspirations to write this paper.

As it was proved in \cite{WR} for $d=2$ and any distribution, and in \cite{K} for arbitrary $d$ under mild conditions on the distribution, with probability 1 there exists an optimal path between any two lattice points, which is called a geodesic. Furthermore, if the probability distribution function is continuous, geodesics are unique with probability 1. In \cite{M} we verified by a natural argument that the topological model has the same properties in a residual set, if the set $A$ has no isolated points, which somewhat corresponds to the continuity of the distribution function. However, the existence and uniqueness of infinite geodesics proved to be a more difficult and interesting problem. We call an infinite path indexed by $\mathbb{N}$ a geodesic ray if each of its finite subpaths are finite geodesics. As finite geodesics exist between any pair of lattice points with probability 1 or in a residual set respectively in the two setups, by K\H{o}nig's lemma one can easily infer that any point is a starting point of an infinite geodesic with probability 1 or in a residual set. Now it is a natural question whether there are more distinct geodesic rays, where by distinct we mean that they share only finitely many edges. 

In the probabilistic setup it is conjectured that for continuous distributions there are infinitely many of them with probability 1. For $d=2$ and a certain class of distribution functions this claim was verified in \cite{AD}. In the topological setup, we certainly encounter a different behaviour: in \cite{M} we proved that in a residual subset of configurations there are at most distinct $4d^2+1$ geodesic rays. Moreover, if $\sup A> 5\inf A$ and $d\geq{2}$, we have exactly one geodesic ray in a residual subset of $\Omega$. (When $d=1$, both of the halflines are geodesic rays trivially, hence it is reasonable to consider $d\geq{2}$.) However, we could not answer whether there may exist distinct geodesic rays in a second category set for arbitrary nontrivial $A$, that is for $A$s such that the cardinality of $A$ is larger than 1. The following sharp version of the cited theorems cures this deficiency: \begin{theorem}
Assume $d\geq{2}$. If $A$ is nontrivial, that is it has cardinality larger than 1, then in a residual subset of $\Omega$ there exists exactly one geodesic ray.
\end{theorem}

This theorem leads to an interesting remark: in a residual set, any lattice point is the starting point of a geodesic ray, while there are no distinct geodesic rays. Thus any two of these geodesic rays merge after a finite number of edges.

Another topic considered in our earlier paper was the topological correspondant of the Cox--Durrett limit shape theorem. Let us denote by $B(t)$ the ball of radius $t$ centered at the origin in the pseudometric $T$, that is the subset of $\mathbb{R}^d$ we might reach from the origin in time $t$. A truly interesting result of the theory (see \cite{CD}) is that there exists a deterministic limit shape $B_\mu$, which has the property that as $t$ tends to $+\infty$, with probability one $\frac{B(t)}{t}$ tends to $B_\mu$ in some sense. Various works can be found in the literature based on this theorem about the speed of this convergence for example. We might ask if a similar statement holds in a residual set in the topological setup. Let us denote by $D_r$ the $\ell_1$ closed ball of radius $r$ centered at 0, and let $\mathcal{K}_A^d$ be the set of connected compact sets in $\mathbb{R}^d$ satisfying
\begin{displaymath}
D_\frac{1}{\sup A} \subseteq K \subseteq D_\frac{1}{\inf A},
\end{displaymath}
\noindent where the leftmost set is replaced by $\{0\}$ if $\sup A =\infty$, and the rightmost set is replaced by $\mathbb{R}^d$ if $\inf A = 0$. Furthermore, we say that $K\in\mathcal{P}_A^d$ if $K\in\mathcal{K}_A^d$, and for each $x\in{K}$ there is a "topological path" in $K$ of $\ell_1$-length at most $\frac{1}{\inf A}$ from $0$ to $x$. (From now on, we use the terms path and topological path in order to clearly distinguish paths in graph theoretical sense and paths in topological sense.) The closure of $\mathcal{P}_A^d$ in $\mathcal{K}_A^d$ with respect to the Hausdorff metric is simply denoted by $\overline{\mathcal{P}_A^d}$. In \cite{M} we proved

\begin{prop}
If $\frac{B(t_n)}{t_n}\to K$ in the Hausdorff metric in some configuration for some compact set $K$, where $(t_n)_{n=1}^{\infty}$ is a sequence diverging to $+\infty$, then $K\in\overline{\mathcal{P}_A^d}$. Moreover, if $\inf A =0$, or $\sup A=\infty$, then in a residual subset of $\Omega$ for any $K\in\overline{\mathcal{P}_A^d}$ there exists a sequence $(t_n)_{n=1}^{\infty}$ which tends to $+\infty$ and
\begin{displaymath}
\frac{B(t_n)}{t_n} \to K
\end{displaymath}
\noindent in the Hausdorff metric.
\end{prop}

We note here that in \cite{M} we gave a slightly different and more complicated definition for $\mathcal{P}_A^d$, but these slightly different families have the same closures in $\mathcal{K}_A^d$. Thus from the view of the previous proposition and the upcoming arguments it is convenient to use this less elaborate definition.

The characterization in the case when $A$ is bounded away both from $0$ and $\infty$ remained an open question. A conjecture on a sufficient condition about $K$ was mentioned previously as a sidenote: any convex set in $\mathcal{K}_A^d$ should arise as a limit set. As the convex sets of $\mathcal{K}_A^d$ are also in $\mathcal{P}_A^d$, it is consistent with the first statement of Proposition 1.2 at least. In Section 3 we will verify the following general theorem which proves this conjecture:
\begin{theorem}
Let $A\subseteq[0,+\infty)$ be arbitrary. In a residual subset of $\Omega$ for any convex set $K\in\mathcal{K}_A^d$ there exists a sequence $(t_n)_{n=1}^{\infty}$ which tends to $+\infty$ and
\begin{displaymath}
\frac{B(t_n)}{t_n} \to K
\end{displaymath}
\noindent in the Hausdorff metric.
\end{theorem}

The other motivation for this paper was a question raised by Kornélia Héra after a talk given about the earlier results. Roughly, her question was what can be said if we consider $A\subseteq\mathbb{C}$ instead of $A\subseteq\mathbb{R}$ and we define appropriate absolute values as passage times. We will consider a more general setup which we will refer to as Hilbert first passage percolation or Hilbert percolation: we fix a real Hilbert space $\mathcal{H}$ and $A\subseteq{\mathcal{H}}$. The passage vector $v(e)$ of any edge will be an element of $A$, while the passage time of the edge is $\tau(e)=\|v(e)\|_{\mathcal{H}}$. The passage vector $v(\Gamma)$ of a path $\Gamma$ is defined as the sum of passage vectors of the contained edges, and the passage time $\tau(\Gamma)$ of the path is the norm of the passage vector of the path. The passage set between any two points $x,y\in\mathbb{Z}^d$ is the set $S(x,y)$ of passage vectors of paths connecting them. Finally, the passage time between any two points $x,y\in\mathbb{R}^d$ is
\begin{displaymath}
T(x,y)=\inf_{\Gamma}\tau(\Gamma),
\end{displaymath}
\noindent where the infimum is taken over all the paths connecting $x'$ to $y'$, where $x'$ and $y'$ are the unique lattice points such that $x\in x'+[0,1)^d$, $y\in y'+[0,1)^d$. In other words,
\begin{displaymath}
T(x,y)=\inf\{\|v\|_{\mathcal{H}}:v\in S(x,y)\}.
\end{displaymath}

We note here that various generalizations of the probabilistic setup also exist, originally motivated by the setup with non-i.i.d edge weights. In \cite{Bo}, the general stationary and ergodic case is considered, and an analogue of the Cox--Durrett theorem is proved for instance. For the most general version known see \cite{B}. As it is quite involved and not of our direct interest, we do not go into details.

The theory of Hilbert first passage percolation is obviously a generalization of the ordinary topological first passage percolation, as in the case $\mathcal{H}=\mathbb{R}$ and $A\subseteq{[0,\infty)}$ the new definition of passage times coincide with the old one. However, it is useful to note that if $A$ contains negative values, we face a slight ambiguity as we do not have this coincidence. Anyway, we should not worry about this phenomenon as the original topological model collapses in that case and does not deserve much attention.

When it comes to questions of residuality, it is usual to restrict ourselves to separable spaces for technical reasons. As a consequence, henceforth we assume that all the real Hilbert spaces involved are separable.

In Section 4 we will restrict further our scope to finite dimensional spaces, define strongly positively dependent sets, and prove that if $A$ is bounded and strongly positively dependent then the Hilbert first passage percolation gets trivial in some sense. We also consider the question how common are strongly positively dependent sets amongst the compact sets in terms of Baire category.

The purpose of Section 5 is to examine the geometric behaviour of the Hilbert percolation in the generic case. In this section  More explicitly, we will define optimal paths and geodesics in the Hilbert first passage percolation and focus on their their relationship. An optimal path is a geodesic if each of its subpaths is also an optimal path. There are two natural properties we will expect to get a quite tame geometric structure: for any $\Gamma'\subseteq \Gamma$ we have that $\tau(\Gamma')\leq \tau(\Gamma)$ and there is a geodesic between any pair of lattice points. If $d=1$, the latter property obviously holds if the former one does, hence we will assume $d\geq{2}$. Through Lemma 5.2 and Lemma 5.3 we will prove the following theorem displaying that if these properties are present in the generic case and $\Omega$ is a Baire space, that is residual subsets intersect any nontrivial open set, then the Hilbert first passage percolation we consider is equivalent to some ordinary topological first passage percolation. This phenomenon provides one more reason (beside of simplicity) supporting that we are primarily interested in the ordinary topological first passage percolation.
\begin{theorem}
Let $A\subseteq{\mathcal{H}}$ such that $\Omega$ is a Baire space and let $d\geq{2}$. Assume that in a residual subset of $\Omega$ for any $\Gamma'\subseteq \Gamma$ we have that $\tau(\Gamma')\leq \tau(\Gamma)$ and there is a geodesic between any pair of lattice points. Then $A$ is contained by a ray, that is it is linearly isomorphic to a subset of $[0,+\infty)$. \end{theorem}
We note that if for example $A\subseteq\mathcal{H}$ is $G_{\delta}$, then $\Omega$ is a Baire space, hence the above theorem holds in quite natural cases. Indeed, $\mathcal{H}$ is obviously Polish as $\mathcal{H}$ is separable, hence $A$ is also Polish due to Alexandrov's theorem. Consequently, $\Omega$ is also Polish, as a countable product of Polish spaces. Thus $\Omega$ is Baire due to Baire category theorem. (For details, see \cite{Ku} for instance.) We also emphasize that it is rather usual to consider Baire spaces exclusively when one studies Baire category as otherwise residual subsets fail to grasp the concept of being large. However, our other results are correct anyway, it is the reason why we did not make this assumption previously.

In the ordinary topological first passage percolation we proved that in the generic configuration for any point $x$ there exists a geodesic ray such that $x$ is its starting point, and the proof was based on the fact that there exists a finite geodesic between any pair of lattice points in the generic configuration. Given the previous theorem, it is natural to ask how many distinct geodesic rays may exist generically in the Hilbert percolation, and whether it is possible that there are no geodesic rays at all. The following theorem which is proven in Section 6 is naturally analogous to Theorem 1.1:
\begin{theorem}
Assume $d\geq{2}$. If $A\subseteq\mathcal{H}$ is nontrivial, that is it has cardinality larger than 1, then in a residual subset of $\Omega$ there exists at most one geodesic ray.
\label{hilbertrays}
\end{theorem}
At the end of this section we also provide an example in which there are no geodesic rays at all generically.

In Section 7 we formulate a few open questions motivated by this paper.

\section{Uniqueness of the geodesic ray}

We declare at this point how we will think about the topology on $\Omega$. The most convenient way for us is to consider cylinder sets as the basis of the topology, that is the basis sets are of the form
\begin{displaymath}
U=\times_{e\in E}U_e,
\end{displaymath}
\noindent where each $U_e$ is open in $A$ and with at most finitely many exceptions $U_e=A$. We say that $U_e$ is the projection of $U$ to the edge $e$.

We use the notation $|x|$ for the $\ell_1$-norm of $x\in\mathbb{R}^d$ throughout the paper.

\begin{proof}[Proof of Theorem 1.1] The outline of the proof is similar to the ones of its weaker counterparts, however, it relies on a more elaborate geometric construction. On the other hand, we  somewhat simplified technical details in general.

As earlier, first we will prove that if $x$ is a fixed lattice point then apart from a meager subset of $\Omega$ there is no more than one geodesic ray starting from $x$. Clearly it suffices to prove this claim for $x=0$. Let $F(0)$ denote the set of configurations in which there are at least two distinct geodesic rays starting from the origin. Then $F(0)=\bigcup_{m=1}^{\infty}F_m(0)$ where $F_m(0)$ stands for the set of configurations in which there are at least two distinct geodesics starting from the origin such that they have at most $m$ edges in common. We claim that for any $m$ we have that $F_m(0)$ is a nowhere dense set in $\Omega$, which would verify our preliminary statement about the meagerness of $F(0)$.

Fix $U$ to be a cylinder set, and denote the set of edges belonging to nontrivial projections of $U$ by $E_U=\{e_1, e_2, ..., e_N\}$. We can simply construct a smaller cylinder set $U'$ by shrinking the projections $U_{e_1}, ..., U_{e_N}$, such that all of these projections are bounded in $\mathbb{R}$. Then for any configuration in $U'$, the sum of passage times over the edges $e_1,...,e_N$ is bounded by a constant $C$. The novelty appears at this point: instead of considering concentric hypercubes centered at the origin, we take a skew construction. More precisely, let $K_1=[-p,p]^d$ and $K_2=[-q',q]\times[-r,r]^{d-1}$ where $p,q,q',r\in\mathbb{N}$ and $p$ is chosen such that the edges in $E_U$ are in the interior of $K_1$. The values $q<r<q'$ are to be fixed later. Let us denote the set of edges in $K_2$ which are not in the interior of $K_1$ by $E^*$. We will define $V\subseteq{U'}$ as a cylinder set which has nontrivial projections to the edges in $E_U\cup E^*$. The underlying concept is borrowed from the proof given for the case $\sup A < 5 \inf A$: for the configurations in $V$ we would like to have essentially one (and the same) geodesic from the boundary $\partial K_1$ to the boundary $\partial K_2$, notably the line segment connecting $p\xi_1$ and $q\xi_1$ (in general $\xi_i$ denotes the $i$th coordinate vector). By this we mean that for any lattice points $x_1\in \partial K_1$ and $x_2 \in\partial K_2$, a geodesic $\Gamma$ from $x_1$ to $x_2$ eventually arrives in $p\xi_1$, and then it goes along the line segment $[p\xi_1,q\xi_1]$. It would be sufficient: any geodesic ray starting from the origin eventually leaves $K_1$ and $K_2$, and a geodesic ray is a geodesic between any two of its points, the previous properties would guarantee that any geodesic ray starting from the origin would go along the line segment $[p\xi_1,q\xi_1]$. However, that would mean that our configuration cannot be in $F_m(0)$ for $q-p>m$ as there would not exist at least two distinct geodesics starting from the origin such that they have at most $m$ edges in common.

Let us make the above argument rigorous. Fix $a<b$ in $A$. Moreover, fix $\varepsilon>0$ and $\lambda>1$ such that $(a+\varepsilon)\lambda<b-\varepsilon$ holds. Finally, for later usage define a small value $\varepsilon_e$ for each edge $e\in{E}$ such that $\sum_{e\in{E}}\varepsilon_e<\varepsilon$.  We will have small passage times on the edges of $\partial K_1$, $\partial K_2$, and along the line segment $[p\xi_1,q\xi_1]$ to guarantee a path with considerably low passage time between any two points of $\partial K_1$ and $\partial K_2$. We call these edges cheap. Meanwhile on other edges between the two boundaries (e.g. the expensive edges) we would like to have considerably larger passage times. Thus for every cheap edge $e$ we define the relatively open set 
\begin{displaymath}
V_e=(a-\varepsilon_e, a + \varepsilon_e ) \cap A,
\end{displaymath}
and for any expensive edge we define
\begin{displaymath}
V_e=(b-\varepsilon_e, b+ \varepsilon_e ) \cap A.
\end{displaymath}
By this, we have defined $V$. Now consider any configuration in $V$. For technical convenience we will prove the following claim, which is formally stronger than what we stated before: if $x_1\in\partial K_1 \cup [p\xi_1,(q-1)\xi_1]$, while $x_2\in [(p+1)\xi_1,q\xi_1] \partial K_2$, then there is no geodesic from $x_1$ to $x_2$ which uses expensive edges. Proceeding towards a contradiction, assume the existence of $x_1,x_2$, and a geodesic $\Gamma$ which refutes this claim. As any subpath of a geodesic is also a geodesic, we can assume that $\Gamma$ uses expensive edges only by passing to a suitable subpath. Hence the passage time of $\Gamma$ can be estimated from below by
\begin{displaymath}
\tau(\Gamma)\geq |x_2-x_1|b-\varepsilon\geq|x_2-x_1|(b-\varepsilon).
\end{displaymath}
We will arrive at a contradiction by constructing a cheaper $\Gamma'$ from $x_1$ to $x_2$ which does not use expensive edges,  In the following we will separate cases based on the position of $x_2$. The figure below displays how we will construct $\Gamma'$ with the desired properties in one of the cases and it also helps understanding the other constructions.

\begin{figure}[h!]
  \includegraphics[width=300pt]{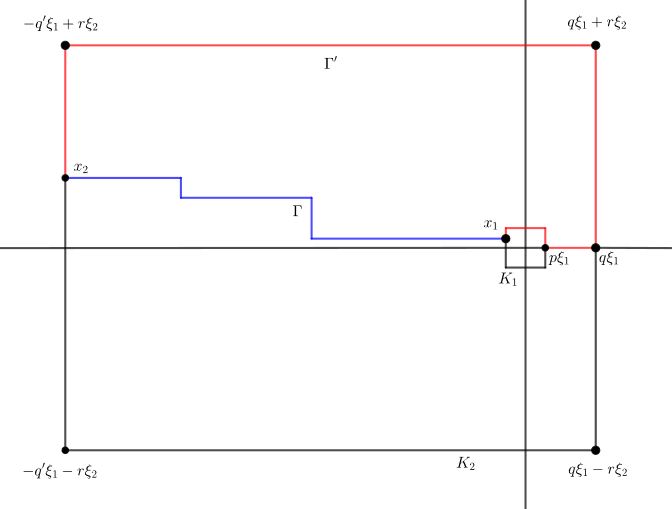}
  \caption{The case when $x_2\in\partial K_2$ is on the facet containing $-q'\xi_2$ and $d=2$.}
  \label{geodesics}
\end{figure}

If $x_2\in[p\xi_1,q\xi_1]$, we have a simple task. Indeed, in this case $\Gamma$ can be replaced by a path $\Gamma'$ which is not longer in $\ell_1$, and instead of using expensive edges only, it uses cheap edges exlusively. Thus $\tau(\Gamma')<\tau(\Gamma)$ clearly holds.

Now assume that $x_2$ is on the same facet of $\partial K_2$ as $q\xi_1$. We separate two subcases:
\begin{itemize}
\item $x_1\in[p\xi_1,(q-1)\xi_1]$. In this case there exists an $\ell_1$-optimal path from $x_1$ to $x_2$ which consists of cheap edges exclusively. Choose such a path to be $\Gamma'$, it contains $|x_2-x_1|$ edges. The passage time of $\Gamma'$ can be estimated from above by
\begin{displaymath}
\tau(\Gamma')\leq |x_2-x_1|a+\varepsilon\leq|x_2-x_1|(a+\varepsilon).
\end{displaymath}
Comparing the bounds gives
\begin{displaymath}
|x_2-x_1|(b-\varepsilon) \leq |x_2-x_1|(a+\varepsilon).
\end{displaymath}
However, as $a+\varepsilon<b-\varepsilon$, it is clearly impossible, thus we have handled this case.
\item $x_1\in\partial K_1$. In this case we define $\Gamma'$ by joining together two shorter paths $\Gamma_1$ and $\Gamma_2$. The path $\Gamma_1$ will run on $\partial K_1$ from $x_1$ to $p\xi_1$ such that it uses as few edges as it is possible. Consequently, $|\Gamma_1|\leq 2dp$. Meanwhile the path $\Gamma_2$ will go from $p\xi_1$ to $x_2$ such that it is $\ell_1$-optimal and uses cheap edges exclusively. (By the choice of the facet containing $x_2$ it is clearly possible.) Now by the triangle inequality we have
\begin{displaymath}
|\Gamma_2|\leq|x_2-x_1|+2dp.
\end{displaymath}
Using the estimate for the number of edges in $\Gamma_1$ and $\Gamma_2$ we can obtain an upper bound for the passage time of $\Gamma'$:
\begin{displaymath}
\tau(\Gamma')\leq(|x_2-x_1|+4dp)a+\varepsilon\leq(|x_2-x_1|+4dp)(a+\varepsilon).
\end{displaymath}
Comparing the bounds gives
\begin{displaymath}
|x_2-x_1|(b-\varepsilon)\leq(|x_2-x_1|+4dp)(a+\varepsilon).
\end{displaymath}
As $(a+\varepsilon)\lambda<b-\varepsilon$ now we can obtain after division
\begin{displaymath}
\lambda|x_2-x_1|\leq|x_2-x_1|+4dp,
\end{displaymath}
or equivalently,
\begin{displaymath}
|x_2-x_1|\leq\frac{4dp}{\lambda-1}.
\end{displaymath}
However, $|x_2-x_1|\geq q-p$ necessarily holds. Hence if $q$ is chosen to be sufficiently large compared to $p$ we get a contradiction, which concludes this case.
\end{itemize}

Assume now that $x_2$ lies on a facet of $K_2$ neighboring the one containing $q\xi_1$. In this case we construct $\Gamma'$ by joining at most three paths $\Gamma_1,\Gamma_2,\Gamma_3$: if $x_1\in\partial K_1$ we define $\Gamma_1$ in order to reach $p\xi_1$ as in the second subcase of the previous case. Next we use $\Gamma_2$ to reach $q\xi_1$ using the edges of $[p\xi_1,q\xi_1]$. Finally, we define $\Gamma_3$ to reach $x_2$ so that it is optimal in $\ell_1$ and uses only the edges of $\partial K_2$. The first two parts use at most $2dp+q-p$ edges, while we can get a simple upper estimate for the length of $\Gamma_3$ using the triangle inequality, notably
\begin{displaymath}
|\Gamma_3|\leq|x_2-x_1|+2dp+q-p.
\end{displaymath}
Consequently, we have
\begin{displaymath}
|\Gamma'|\leq|x_2-x_1|+4dp+2q-2p.
\end{displaymath}
As all these edges are cheap, we deduce the following bound:
\begin{displaymath}
\tau(\Gamma')\leq(|x_2-x_1|+4dp+2q-2p)a+\varepsilon\leq(|x_2-x_1|+4dp+2q-2p)(a+\varepsilon).
\end{displaymath}
Comparing to the lower bound given for $\tau(\Gamma)$ we gain
\begin{displaymath}
|x_2-x_1|(b-\varepsilon)\leq(|x_2-x_1|+4dp+2q-2p)(a+\varepsilon).
\end{displaymath}
Given the ratio bound on $a+\varepsilon$ and $b-\varepsilon$ it yields
\begin{displaymath}
\lambda|x_2-x_1|\leq|x_2-x_1|+4dp+2q-2p.
\end{displaymath}
Thus, simple rearrangement yields
\begin{displaymath}
|x_2-x_1|\leq\frac{4dp+2q-2p}{\lambda-1}.
\end{displaymath}
The right hand side expression is already fixed, while for the left hand side we have $|x_2-x_1|\geq r-q$. Thus if $r$ is chosen so that it is sufficiently large compared to the already fixed $q$, then we get a contradiction, which concludes this case.

The final case to consider is when $x_2$ is on the same facet of $K_2$ as $-q'\xi_1$. In this case we define $\Gamma'$ as the union of at most four shorter paths $\Gamma_1,\Gamma_2,\Gamma_3,\Gamma_4$. Concerning $\Gamma_1$ and $\Gamma_2$ we resort to the previous case in order to get to $q\xi_1$ from $x_1$, using at most $2dp+q-p$ cheap edges. Then we define $\Gamma_3$ to be the line segment $[q\xi_1,q\xi_1+r\xi_2]$, thus we reach a facet neighboring to the one containing $x_2$ using $r+2dp+q-p$ cheap edges. Finally we define $\Gamma_4$ to reach $x_2$ so that it is optimal in $\ell_1$ and uses only the edges of $\partial K_2$. By the triangle inequality we have
\begin{displaymath}
|\Gamma_4|\leq|x_2-x_1|+r+2dp+q-p,
\end{displaymath}
and hence
\begin{displaymath}
|\Gamma'|\leq|x_2-x_1|+2r+4dp+2q-2p.
\end{displaymath}
As all these edges are cheap, we deduce the following bound:
\begin{displaymath}
\tau(\Gamma')\leq(|x_2-x_1|+2r+4dp+2q-2p)a+\varepsilon\leq(|x_2-x_1|+2r+4dp+2q-2p)(a+\varepsilon).
\end{displaymath}
Comparing to the lower bound given for $\tau(\Gamma)$ we gain
\begin{displaymath}
|x_2-x_1|(b-\varepsilon)\leq(|x_2-x_1|+2r+4dp+2q-2p)(a+\varepsilon).
\end{displaymath}
By the ratio bound on $a+\varepsilon$ and $b-\varepsilon$ it yields
\begin{displaymath}
\lambda|x_2-x_1|\leq|x_2-x_1|+2r+4dp+2q-2p.
\end{displaymath}
Simple rearrangement yields
\begin{displaymath}
|x_2-x_1|\leq\frac{2r+4dp+2q-2p}{\lambda-1}.
\end{displaymath}
The right hand side expression is already fixed, while for the left hand side we have $|x_2-x_1|\geq q'-p$. Thus if $q'$ is chosen so that it is sufficiently large compared to $r$, then we get a contradiction, which concludes this case, and also the proof of the fact that $F(0)$ is meager.

The final step of the proof does not differ at all from the final step of the proof given for the case $5\inf A < \sup A$. Namely, let $F\subseteq{\Omega}$ be the set of configurations in which there are at least two geodesic rays. Moreover, let $F(x)$ be the set of configurations in which there are at least two distinct geodesic rays with starting point $x$, and $F_m$ be the set of configurations in which there exist two disjoint geodesic rays with starting point in the cube $[-m,m]^d$. Then 
\begin{displaymath}
F=\left(\bigcup_{x\in\mathbb{Z}^d}F(x)\right)\cup\left(\bigcup_{m=1}^{\infty}F_m\right)
\end{displaymath}
\noindent holds: if there exist at least two geodesic rays they are either disjoint or have a common point $x$, and in the latter case we have two geodesic rays starting from $x$ if we forget about the initial parts of these geodesics. Furthermore, we know that each of the sets $F(x)$ are meager by our argument up to this step. Thus if we could obtain that each $F_m$ is nowhere dense, that would conclude the proof. However, having seen the proof of the first part we do not have a difficult task as we can basically repeat that argument. Indeed, in that proof we showed that for a given cylinder set $U$ one can construct boxes $K_1,K_2$ and another cylinder set $V\subseteq{U}$ such that for configurations in $V$ any geodesic from $\partial K_1$ to $\partial K_2$ goes along the line segment $[p\xi_1,q\xi_1]$. Thus if we choose $p>m$ during the construction we will obtain that none of the configurations in $F_m$ can appear in $V$ as in $V$ there cannot be two disjoint geodesic rays starting from $[-m,m]^d$, given they all meet in $p\xi_1$. Thus $F_m$ is nowhere dense indeed, which concludes the proof of the theorem. \end{proof}

As we already pointed out in the introduction, geodesic rays starting from any two distinct points meet after a finite number of edges, which is a rather interesting, extraordinary behaviour. It is useful to think a bit about how we should imagine the unique geodesic ray then, what it looks like. It is tempting to imagine a picture in which it has some asymptotic direction. However, using similar techniques to the one seen in the proof, one can easily verify that generically the unique geodesic ray intersects any path of infinite length infinitely many times. Consequently, it is more appropriate to think about it as an infinite path which looks somewhat spiralic in the long run.

We note that in each case considered in the proof, in the estimates we only needed that the $\Gamma'$ we construct is optimal in $\ell_1$ between its endpoints amongst paths contained by $\partial K_1 \cup [p\xi_1,q\xi_1]\cup\partial K_2$, while $\Gamma$ does not use any edge contained by this set. We will refer back to this remark in the proof of Theorem \ref{hilbertrays}.

\section{The asymptotic behaviour of $\frac{B(t)}{t}$}

In the following we will assume that $A$ is bounded away both from $0$ and $+\infty$ since the other cases are covered by Proposition 1.2 as the convex sets of $\mathcal{K}_A^d$ are in $\mathcal{P}_A^d$. What we gain by this assumption is that we circumvent certain technical difficulties, however, we note that some of the definitions and results could be generalized to these extremes.

We start our investigations by introducing two families of metrics in $\mathbb{R}^d$:

\begin{defi}
Let $f$ be a nonnegative, measurable function. The pseudometric $d_{f,\ell_1}$ induced by $f$ is defined by
\begin{displaymath}
d_{f,\ell_1}(x,y)=\inf_{\Gamma:x\to y} \int_{\Gamma}f(t)\mathrm{d}s,
\end{displaymath}
where the arc length is considered in $\ell_1$, and we consider piecewise linear topological paths with finitely many pieces upon taking infimum.
\end{defi}

\begin{defi}
We call a pseudometric $\rho$ on $\mathbb{R}^d$ a percolation metric with support $A$ if there exists a measurable function $f:\mathbb{R}^d\to\overline{A}$ so that $\rho=d_{f,\ell_1}$.
\end{defi}

In the sequel we omit $\ell_1$ from the subscript as we are only concerned with $\ell_1$ based metrics, and unless it may cause ambiguity we will not write out the suffix "with support $A$" either. The family of sets arising as closed unit balls of percolation metrics, centered at the origin, will be denoted by $\mathcal{W}_A^d$. As $A$ is bounded away from both $0$ and $+\infty$, the elements of $\mathcal{W}_A^d$ are compact sets, each of them is the closure of its interior. Moreover, each percolation metric with support $A$ is a proper metric indeed. The closure of $\mathcal{W}_A^d$ in $\mathcal{K}_A^d$ is denoted by $\overline{\mathcal{W}_A^d}$.

The following theorem is simple to prove and displays how percolation metrics are related to the limit of sequences of the type $\frac{B(t_n)}{t_n}$:

\begin{theorem}
Assume that $\frac{B(t_n)}{t_n}\to K$ in the Hausdorff metric in some configuration for a sequence $t_n$ diverging to $+\infty$. Then $K\in \overline{\mathcal{W}_A^d}$.
\end{theorem}

\begin{megj}
Due to Proposition 1.2, this theorem yields $\overline{\mathcal{W}_A^d}\subseteq\overline{\mathcal{P}_A^d}$.
\end{megj}

\begin{proof}[Proof of Theorem 3.3]
Consider the subgraph $\tilde{B}(t_n)$ of $\mathbb{Z}^d$ accessible in time $t_n$ from the origin. We obviously have $\frac{\tilde{B}(t_n)}{t_n}\to K$ by assumption. Now we define $f_n$ as follows: in the relative interior of an edge of the graph $\frac{\mathbb{Z}^d}{t_n}$ we define $f_n$ to have the same value as the passage time of the corresponding edge in $\mathbb{Z}^d$. (In the endpoints this definition would not give a unique value, but the value on a discrete set will not have any importance anyway.) For any remaining point $x\in\mathbb{R}^d$ we set $f_n(x)=\sup A$. We state that the Hausdorff distance of the closed unit ball $B_n$ of $d_{f_n}$ and $\frac{\tilde{B}(t_n)}{t_n}$ converges to 0: as $B_n\in\mathcal{W}_A^d$ that would conclude the proof. As the containment $\frac{\tilde{B}(t_n)}{t_n}\subseteq B_n$ is obvious, we only have to examine how far a point of $B_n$ can lie from the graph $\frac{\tilde{B}(t_n)}{t_n}$. We also note that the value of $f_n$ only matters in $D_{\frac{1}{\inf A}}$, as $B_n\subseteq D_{\frac{1}{\inf A}}$ necessarily holds.

First let us notice that if $x$ is a vertex of $\frac{\mathbb{Z}^d}{t_n}$ then in the definition of $d_{f_n}(0,x)$ it suffices to consider topological paths which are also paths in the graph $\frac{\mathbb{Z}^d}{t_n}$. Indeed, by definition for any $\varepsilon>0$ there exists a topological path $\Gamma$ from the origin to $x$ so that we have
\begin{displaymath}
\int_{\Gamma}f_n(t)\mathrm{d}s<d_{f_n}(0,x)+\varepsilon.
\end{displaymath}
Our aim is to show that there exists a $\Gamma'$ which is a path in the graph $\frac{\mathbb{Z}^d}{t_n}$ and the integral of $f_n$ on $\Gamma'$ does not exceed the integral of $f_n$ on $\Gamma$. If $\Gamma$ itself is such a path then we are done. Moving towards a contradiction, assume that there is a point $x$ and a path $\Gamma$ connecting the origin to $x$ which is not such a path, and there is no such $\Gamma'$. Let $N(\Gamma)$ be the number of pieces of $\Gamma$ in the graph $\frac{\mathbb{Z}^d}{t_n}$, where by such pieces we mean largest connected components in one of the edges of $\frac{\mathbb{Z}^d}{t_n}$. Meanwhile let $M(\Gamma)$ be the number of complementary components of $\Gamma$. Now choose a contradictory $x,\ \Gamma$ so that $M(\Gamma)$ is minimal, and amongst these one for which $N(\Gamma)$ is minimal. As $\Gamma$ is not a path in the graph $\frac{\mathbb{Z}^d}{t_n}$, we can choose line segments $[y,y']$ and $[z,z']$ so that they are consecutive pieces in the above sense, that is they are respectively contained by some edges $\frac{e_y}{t_n}, \frac{e_z}{t_n}$, and the subpath $\Gamma_1$ of $\Gamma$ between $y'$ and $z$ might hit the graph $\frac{\mathbb{Z}^d}{t_n}$ in a discrete set only. Now choose $\Gamma'$ to be a modification of $\Gamma$: from the origin to $y$ and from $z$ to $x$ we do not alter $\Gamma$, but we replace some of the remaining parts based on the relation between $\tau(e_y)$ and $\tau(e_z)$ and the relative position of these edges. Notably, unless $e_y$ is an orthogonal translated image of $e_z$, it is simple to check that there is an $\ell_1$-optimal topological path $\Gamma_1'$ from $y'$ to $z$ which is in fact a path in the graph $\frac{\mathbb{Z}^d}{t_n}$. Moreover, as $\Gamma_1$ does not hit the graph $\frac{\mathbb{Z}^d}{t_n}$, we know $f_n=\sup A$ in $\Gamma_1$ almost everywhere, while it is at most $\sup A$ in $\Gamma_1'$. Hence, if we define $\Gamma'$ as the topological path gained from $\Gamma$ by replacing $\Gamma_1$ with $\Gamma_1'$, we reduce the $d_{f_n}$-length, thus $\Gamma'$ also has to be a contradictory topological path from the origin to $x$. (Or it is already a path in the graph, which would also be a contradiction.) However, we reduced $M(\Gamma)$ by one, which contradicts the choice of $x,\ \Gamma$. Thus we have a contradiction in the case when $e_y$ cannot be obtained from $e_z$ by an orthogonal translation.

Let us consider the other case. Let us also assume $\tau(e_y)\leq\tau(e_z)$, in the other case we can use the same argument by symmetry. In this case we proceed the following way: we project orthogonally $z'$ to $\frac{e_y}{t_n}$ to gain $z^*$, and we gain $\Gamma'$ by replacing the subpath of $\Gamma$ from $y$ to $z'$ by $[y,z^*]\cup[z^*,z']$. Now 
\begin{displaymath}
\left|y-z^*\right|\leq|y'-y|+|z'-z|,
\end{displaymath}
hence by $\tau(e_y)\leq\tau(e_z)$ we have that the integral of $f_n$ on $[y,z^*]\subseteq\Gamma'$ cannot exceed the integral of $f_n$ on $[y,y']\cup[z,z']\subseteq\Gamma$. On the other hand, the part of $\Gamma'$ connecting the edges $\frac{e_y}{t_n}$ and $\frac{e_z}{t_n}$ is $\ell_1$-optimal, hence the integral of $f_n$ here cannot exceed the integral of $f_n$ on the corresponding part of $\Gamma$, either. Consequently, $\int_{\Gamma'}f_n\leq\int_{\Gamma}f_n$, thus $\Gamma'$ also has to be a contradictory topological path from the origin to $x$. (Or it is already a path in the graph, which would also be a contradiction.) However, we eliminated the piece $[z,z']$ of $\Gamma$, hence we reduced $N(\Gamma)$ by one while not increasing $M(\Gamma)$. (Except for the case when $z^*,\ z'$ are both vertices of the graph $\frac{\mathbb{Z}^d}{t_n}$, that is $[z^*,z']$ is the union of a few consecutive edges, but in this case, we reduce $M(\Gamma)$ by the previous step.) It contradicts the choice of $x,\ \Gamma$. Thus the claim of the previous paragraph holds.

Now let $x\in B_n$ arbitrary. Consider the smallest lattice hypercube of $\frac{\mathbb{Z}^d}{t_n}$ containing $x$. Denote an arbitrary vertex of it by $x'$. Now $|x-x'|\leq dt_n^{-1}$, which simply yields that $d_{f_n}(x,x')\leq dt_n^{-1} \sup A$. Consequently,
\begin{displaymath}
d_{f_n}(0,x')\leq 1+dt_n^{-1} \sup A
\end{displaymath}
by the triangle inequality. Thus by the claim of the previous paragraph for any $\varepsilon>0$ there exists a path $\Gamma$ from the origin to $x'$ so that it is in the graph $\frac{\mathbb{Z}^d}{t_n}$ and 
\begin{displaymath}
\int_{\Gamma}f(t)\mathrm{d}s<1+dt_n^{-1} \sup A+\varepsilon.
\end{displaymath}
Choose for example $\varepsilon=dt_n^{-1} \sup A$. Now if we go back on $\Gamma$ from $x'$ by $k\leq|\Gamma|$ edges, where $k$ is to be precised later, we get back to a vertex $y$ of $\frac{\mathbb{Z}^d}{t_n}$, and the subpath of $\Gamma$ from $0$ to $y$ guarantees that
\begin{displaymath}
d_{f_n}(0,y)<1+2dt_n^{-1} \sup A-kt_n^{-1} \inf A=1+(2d\sup A - k\inf A)t_n^{-1}.
\end{displaymath}
Thus we can choose $k$ such that the $y$ we obtain satisfies $d_{f_n}(0,y)<1$, and $k$ is bounded by
\begin{displaymath}
k\leq\left\lceil\frac{2d \sup A}{\inf A}\right\rceil.
\end{displaymath}
Hence by the definition of $f_n$ in the graph $\frac{\mathbb{Z}^d}{t_n}$ we have that $y\in\frac{\tilde{B}(t_n)}{t_n}$. However, by the choice of $k$ we have that $|x'-y|\leq\left\lceil\frac{2d\sup A}{\inf A}\right\rceil t_n^{-1}$. Adding it to the upper bound on $|x-x'|$ we obtain
\begin{displaymath}
|y-x|\leq dt_n^{-1} + \left\lceil\frac{2d\sup A}{\inf A}\right\rceil t_n^{-1}.
\end{displaymath}
This quantity is a uniform bound: for any $n$ and $x\in B_n$ we have such a $y\in  \frac{\tilde{B}(t_n)}{t_n}$. Thus as we have that $\frac{\tilde{B}(t_n)}{t_n}\subseteq B_n$, and $t_n\to+\infty$, we obtain that the elementwise Hausdorff distance of these sequences tends to 0. Consequently, $B_n\to K$ as well. As $B_n\in\mathcal{W}_A^d$, it concludes the proof.
\end{proof}

In some sense the above theorem gives a necessary condition on a set $K\in\mathcal{K}_A^d$ being a limit set, however, it would be quite elaborate to check it in any somewhat complicated case. It is more convenient to think about this result as a kind of reformulation of the original definition, whose significance lies in giving a somewhat new perspective, which helps in proving Theorem 1.3 through the construction given in the following lemma:

\begin{lemma}
If $K\in\mathcal{K}_A^d$ is convex, then $K\in\overline{\mathcal{W}_A^d}$.
\end{lemma}

\begin{proof}

It is well-known and easy to check that convex polytopes with rational vertices form a dense subset of convex sets in the Hausdorff metric, and the same argument shows that convex polytopes with rational vertices of $\mathcal{K}_A^d$ form a dense subset of convex sets of $\mathcal{K}_A^d$. Hence it suffices to prove the lemma for $K$ convex polytopes with rational vertices. We can also assume that $K$ has no boundary point in $\partial D_{\frac{1}{\sup A}}$. First we will further assume that $A$ is an interval, or in other words we will allow the metric inducing functions to have values in $[\inf A, \sup A]$ instead of $\overline{A}$.

Fix $\varepsilon>0$ rational with $\varepsilon<\frac{1}{\sup A}$ and take a finite $\varepsilon$-net $H=\{x_1,x_2,...,x_k\}$ of the compact set $\partial K$. These points can be chosen so that each of them  has rational coordinates. By definition and assumption we clearly have $\frac{1}{\sup A} < |x_i| \leq \frac{1}{\inf A}$. Now on the open line segment $I_i=(0,x_i)$ let $f(x)=|x_i|^{-1}$, and in the complement of these line segments let $f(x)=\sup A$. We claim that the closed unit ball $B_f$ of $d_f$ is in $K$. As this unit ball obviously contains the points of $H$, this claim would simply imply that the Hausdorff distance of $B_f$ and $K$ is at most $\varepsilon$, which would yield $K\in\overline{\mathcal{W}_A^d}$.

First we show that if $x\in I_i$ for some $i$ then $[0,x]$ is a $d_f$-optimal topological path. Proceeding towards a contradiction assume that there exists $x \in I_i$ for some $i$ such that there is a shorter topological path $\Gamma$ in $d_f$ from $0$ to $x$ than $[0,x]$. Choose $x$ and $\Gamma$ so that the number of linear pieces of $\Gamma$ is minimal. Amongst such $x$s and $\Gamma$s choose $x$ and $\Gamma$ so that the number of intersected intervals $I_j$ is minimal. As $[0,x]$ is optimal in $\ell_1$, we might assume that $\Gamma$ has a common line segment with one of the intervals $I_j$, as otherwise we have $|\Gamma|\geq |x|$ and  $\left.f\right|_\Gamma\geq \left.f\right|_{[0,x]}$. Besides that we can also assume that $x$ is the first point of $\Gamma$ in $I_i$. Now choose $y$ to be the last point of $\Gamma$ in one of the intervals $I_j$ before $x$. By the choice of $\Gamma$ we can immediately yield that the piece of $\Gamma$ from $0$ to $y$ equals $\Gamma_1=[0,y]$ in fact. Denote the second part of $\Gamma$ by $\Gamma_2$. We know that in $\Gamma_1$ we have $f(t)=|x_j|^{-1}$ while in $\Gamma_2$ we have $f(t)=\sup A$. Moreover, in $[0,x]$ we have $f(t)=|x_i|^{-1}$. Consequently,
\begin{equation}
\int_{\Gamma}f(t)\mathrm{d}s=|y||x_j|^{-1}+|y-x|\sup A,
\end{equation}
while
\begin{equation}
\int_{[0,x]}f(t)\mathrm{d}s=|x||x_i|^{-1}.
\end{equation}
By assumption, we have an inequality between these quantities:
\begin{equation}
|y||x_j|^{-1}+|x-y|\sup A<|x||x_i|^{-1}.
\end{equation}
Our aim right now is to get a contradiction, which we try to achieve by using the convexity of $K$. Note that by $K\in\mathcal{K}_A^d$ we know that in the direction of $x-y$ the shape $K$ contains a segment of $\ell_1$-length $(\sup A) ^{-1}$ starting from the origin, and in the direction of $y$ it contains a segment of $\ell_1$-length $|x_j|$ by definition. Thus we have
\begin{equation}
\frac{x-y}{|x-y|\sup A},\  \frac{y|x_j|}{|y|} \in {K}.
\end{equation}
Let us express $x$ as a positive linear combination of these vectors:
\begin{equation}
x=\frac{x-y}{|x-y|\sup A}\cdot(|x-y|\sup A)+\frac{y|x_j|}{|y|}\cdot\frac{|y|}{|x_j|}.
\end{equation}
The sum of these coefficients might differ from 1, thus multiply both of them by the same scalar to get a convex combination of the original vectors. The convexity of $K$ yields that this vector is also in $K$ by (4):
\begin{equation}
\frac{x}{|x-y|\sup A+|y||x_j|^{-1}}\in K.
\end{equation}
On the other hand, the furthest point of $K$ in the direction of $x$ is the endpoint of $I_i$, that is $x_i=\frac{x}{|x|}|x_i|$. It means
\begin{equation}
\frac{|x_i|}{|x|}\geq\frac{1}{|x-y|\sup A+|y||x_j|^{-1}}.
\end{equation}
Taking the reciprocal of this inequality contradicts (3), thus it verifies the claim about the $d_f$-optimal topological paths to the points of the intervals $I_1,..., I_k$.

This observation yields that to any $x\in{K}$ there exists a $d_f$-optimal topological path $\Gamma$: amongst the ones which do not share segments with any of the intervals $I_1,..., I_k$ the $[0,x]$ line segment is optimal with $d_f$-length $|x|\sup{A}$. On the other hand, amongst the ones which hit any of these intervals we only have to consider the ones which are of the form $[0,y]\cup[y,x]$ for some $y\in{I_j}$ where $[x,y]$ does not intersect any of these intervals, as our previous observation implies. However, for any interval $I_j$ there is an optimal $\Gamma_j$ of these topological paths by a simple compactness argument. Hence the $d_f$-optimal topological path $\Gamma$ to $x$ arises as the optimal one of the path $[0,x]$ and $\Gamma_1,...\Gamma_n$. This argument also shows that what is the closed unit ball $B_f$ of $d_f$: if the optimal topological path $\Gamma$ to $x$ with $d_f$-length at most 1 does not hit any of the intervals $I_1,...,I_k$, then we have $x\in D_{\frac{1}{\sup A}}$, that is in $K$. On the other hand, assume that $\Gamma=\Gamma_j=[0,y]\cup[y,x]$ for some $y\in{I_j}$. Here
\begin{displaymath}
\int_{[0,y]}f(t)\mathrm{d}s=|y||x_j|^{-1}
\end{displaymath}
and 
\begin{displaymath}
\int_{[y,x]}f(t)\mathrm{d}s=|x-y|\sup A.
\end{displaymath}
As the $d_f$-length of the path $\Gamma$ is at most $1$, the sum of these quantities is also at most 1, consequently
\begin{equation}
|x-y|\leq\frac{1-|y||x_j|^{-1}}{\sup A}=:r_y.
\end{equation}
Thus we obtain that $B_f$ contains an $\ell_1$-ball of radius $r_y$ around $y$. However, for $y=0$ this $\ell_1$ ball is contained by $K$ as it equals $D_{\frac{1}{\sup A}}$ and $K\in\mathcal{K}_A^d$. On the other hand, for $y=x_j$ this $\ell_1$ ball is trivial, hence it is also contained by $K$ as $x_j\in{K}$. Consequently, as the function $r_y$ is linear in $[0,x_j]$ and $K$ is convex, we have that the $\ell_1$-ball of radius $r_y$ around $y$ is in $K$ for each $y\in I_j$. 

What we conclude by this argument that $B_f$ equals the union of $D_{\frac{1}{\sup A}}$ and all the $\ell_1$-balls of radius $r_y$ around $y$. Moreover, all the points of $B_f$ are in $K$. We have already seen that it concludes the proof of the lemma in the case when $A$ is an interval.

Now let us consider the case when $A$ is arbitrary. Let $A'=[\inf A, \sup A]$. As $\mathcal{K}_A^d=\mathcal{K}_{A'}^d$, we have $K\in\mathcal{K}_{A'}^d$ as well. Hence we can consider the function $f$ we constructed in the previous case, which induces a metric with unit ball $B_f$ so that its Hausdorff distance from $K$ is at most $\varepsilon$. We are going to replace $f$ by $\tilde{f}$ so that $\tilde{f}$ induces a percolation metric $d_{\tilde{f}}$ which is almost the same as $d_f$, hence its unit ball $B_{\tilde{f}}$ is also close to $K$. As $f$ has values differing from $\sup A$ only in the intervals $I_i$, it suffices to modify $f$ there.

Let $s>0$ be small, to be fixed later. We partition each of the intervals $I_i$ into subintervals of equal length at most $s_i\leq s$, where $\varepsilon$ is an integer multiple of each of the $s_i$s. As all the coordinates and $x_i$ and $\varepsilon$ are rational, we can choose the $s_i$s this way. Such a subinterval $J_i$ will be cut into two further subintervals $J_{i,1}$ and $J_{i,2}$ with length $s_{i,1}$ and $s_{i,2}$ so that if $\tilde{f}=\sup A$ in $J_{i,1}$ and $\tilde{f}=\inf A$ in $J_{i,2}$, then $f$ and $\tilde{f}$ has the same integral in $J_i$, that is
\begin{displaymath}
s_i|x_i|^{-1}=s_{i,1}\sup A + s_{i,2}\inf A.
\end{displaymath}
Our long-term goal is to prove that the Hausdorff distance of $B_{\tilde{f}}$ and $K$ can be arbitrarily small for sufficiently small $s$. There are two things to be checked: we need that $K$ is contained by a small neighborhood of $B_{\tilde{f}}$, and the converse. The first one is simple: by construction it is obvious to see that the integral of $f$ and $\tilde{f}$ equals on any line segment $I_i$, hence $B_{\tilde{f}}$ contains $I_i$. However, the intervals $I_i$ form an $\varepsilon$-net of $K$ which yields that $K$ is contained by the $\varepsilon$-neighborhood of $B_{\tilde{f}}$. The second containment proves to be trickier.

By the choice of $f$ we know that $B_f$ is contained by the $\varepsilon$-neighborhood of $K$. Hence it would be sufficient to verify that $B_{\tilde{f}}$ is contained by a small neighborhood of $B_f$. Our first step in this direction is verifying the following claim: for small $s$, if $x\in{I_i\setminus D_{\varepsilon}}$, then there exists a $d_{\tilde{f}}$-optimal topological path from $\partial D_{\varepsilon}$ to $x$, notably the line segment $J=\left[{x_i}',x\right]\subseteq I_i$, where ${x_i}'=[0,x_i]\cap \partial D_{\varepsilon}$. Proceeding towards a contradiction, assume there is a shorter piecewise linear topological path with finitely many pieces $\Gamma$ in $d_{\tilde{f}}$ to some $x$.  Let $z$ be its starting point. Clearly we can assume that it is the last point of $\Gamma$ in $\partial D_\varepsilon$. Moreover, by an argument similar to the first step of the case when $A$ is an interval, we can assume that $\Gamma=[{x_j}',y]\cup[y,x]$ where $[{x_j}',y]\subseteq I_j$ for some $j\neq{i}$.

First we note that if $|J|$ is very short, that is $x$ is sufficiently close to $D_{\varepsilon}$, then it cannot be possible. Indeed, the set of points ${x_i}'$ form a discrete set, hence between such points there is a minimal $\ell_1$-distance $\delta$. Consequently, $|\Gamma|\geq{\delta}$. Hence if $|J|<\delta\cdot\frac{\inf A}{\sup A}$ then we surely have that the integral of $\tilde{f}$ on $\Gamma$ exceeds its integral on $J$, a contradiction. Thus we can assume that $|J|$ is larger than this bound independent from $s$. Focus only on $s$s smaller than this bound.

By the definition of $\tilde{f}$ it is quite simple to give an upper bound on the integral of $\tilde{f}$ on $J$. Explicitly, in $J$ we have that $\tilde{f}$ equals $\sup A$ and $\inf A$ alternately, and if we have consecutive segments with values $\sup A$ and $\inf A$, then the integral of $\tilde{f}$ on the union of these segments is the same as the integral of $f$. Hence the integral of $\tilde{f}$ might be larger than the integral of $f$ due to the fact that there is one more line segment in which $\tilde{f}$ has value $\sup A$ while $f$ has value $|x_i|^{-1}$ on the complete segment. This segment has length $s_{i,1}$. Consequently, we obtain the following bound
\begin{equation}
\int_{J}\tilde{f}\leq\int_{J}f+s_{i,1}(\sup A- |x_i|^{-1}).
\end{equation}
Similarly, we can give a lower bound on the integral
\begin{equation}
\int_{[{x_j}',y]}\tilde{f}\geq\int_{[{x_j}',y]}f,
\end{equation}
since $s_j$ divides $\varepsilon$, and hence the alternating sequence of line segments with value $\sup A$ and $\inf A$ starts with a complete interval with value $\sup A$ from ${x_j}'$. Consequently, as $\tilde{f}$ equals $f$ almost everywhere in $[y,x]$, we yield
\begin{equation}
\int_{\Gamma}\tilde{f}\geq\int_{\Gamma}f.
\end{equation}
Combining (9) and (11) and the hypothetical inequality between $\int_{J}\tilde{f}$ and $\int_{\Gamma}\tilde{f}$ we conclude
\begin{equation}
\int_{\Gamma}f<\int_{J}f+s_{i,1}(\sup A- |x_i|^{-1}).
\end{equation}
We distinguish three cases based on the relation between $|x_i|$ and $|x_j|$. 

\begin{enumerate}[(i)]

\item Assume that $|x_i|<|x_j|$. As there are finitely many such $j$s, for such $j$s we can choose $r>0$ such that $|x_i|^{-1}>|x_j|^{-1}+r$. Considering what we obtained in the first part, we see
\begin{equation}
\int_{[0,{x_j}']\cup\Gamma}f\geq \int_{[0,{x_i}']\cup J}f.
\end{equation}
By (12) and (13), we have
\begin{equation}
\int_{[0,{x_j}']}f>\int_{[0,{x_i}']}f-s_{i,1}(\sup A- |x_i|^{-1}),
\end{equation}
that is
\begin{equation}
\varepsilon |x_j|^{-1}>\varepsilon |x_i|^{-1}-s_{i,1}(\sup A- |x_i|^{-1})>\varepsilon |x_j|^{-1}+\varepsilon r-s_{i,1}(\sup A- |x_i|^{-1})
\end{equation}
by the choice of $r$. However, for small enough $s$ (and consequently, small enough $s_{i,1}$) it is impossible.

\item Assume that $|x_i|>|x_j|$. For such $j$s we might choose $r>0$ such that $|x_i|^{-1}+r<|x_j|^{-1}$ for all such $j$. Now the integral of $\tilde{f}$ on $J$ is at most $(|x|-\varepsilon)|x_i|^{-1}+s_{i,1}(\sup A- |x_i|^{-1})$, while the integral of $\tilde{f}$ on $\Gamma$ might be estimated from below by
\begin{displaymath}
(|x|-\varepsilon)|x_j|^{-1}>(|x|-\varepsilon)(|x_i|^{-1}+r).
\end{displaymath}
Comparing these bounds we should have
\begin{displaymath}
(|x|-\varepsilon)r<s_{i,1}(\sup A- |x_i|^{-1}).
\end{displaymath}
In this expression the left hand side has a fixed positive bound as we ruled out the possibility of $|J|=|x|-\varepsilon$ being too small. However, the right hand side can be arbitrarily small if we choose $s$ small enough, a contradiction.

\item Finally, assume $|x_i|=|x_j|$. Now by assumption we might choose $r>0$ such that $|x_i|<\sup A - r$. We know that $|y-x|\geq\delta'$ for some $\delta'>0$ as the $\ell_1$-distance between any two distinct segments $[{x_j}',x_j], [{x_i}',x_i]$ is positive, and there are only finitely many such segments. Hence we can deduce $\int_{\Gamma}\tilde{f}\geq{|J||x_i|^{-1}+\delta' r}$, as $|\Gamma|\geq{|J|}$, which is supposed to be smaller then $|J||x_i|^{-1}+s_{i,1}(\sup A- |x_i|^{-1})$. However, it cannot hold if we choose $s$ small enough.

\end{enumerate}

Thus we proved the claim: if $x\in{I_i\setminus D_{\varepsilon}}$, then there exists a $d_{\tilde{f}}$-optimal topological path from $\partial D_{\varepsilon}$ to $x$, notably the line segment $J=\left[{x_i}',x\right]\subseteq I_i$, and hence $d_{\tilde{f}}({x_i}',x)\leq d_f({x_i}',x)+\varepsilon$, if $s$ is sufficiently small. This observation quickly yields that for any $x\in K\setminus D_{\varepsilon}$ we have $d_{\tilde{f}}(\partial D_{\varepsilon},x)\leq d_f(\partial D_{\varepsilon},x)+\varepsilon$, as we have that the $d_{\tilde{f}}$-optimal topological path $\Gamma$ from $\partial D_{\varepsilon}$ to $x$ might have a common line segment with at most one of the segments $[0,x_i]$, thus the integral of $f$ and $\tilde{f}$ on $\Gamma$ might only differ by $\varepsilon$ as the two functions only differ in this piece of $\Gamma$.

Given this fact we can also say something about $d_{\tilde{f}}(0,x)$ for $x\notin D_{\varepsilon}$. (Other $x$s are contained by $B_{\tilde{f}}$ anyway.) A path from $0$ to $x$ must intersect $\partial D_{\varepsilon}$ at some point $y$. Thus we may conclude

\begin{equation}
\begin{split} d_{\tilde{f}}(0,x) & \geq\inf_{y\in\partial D_{\varepsilon}}(d_{\tilde{f}}(0,y)+d_{\tilde{f}}(y,x)) \\ & \geq\inf_{y\in\partial D_{\varepsilon}}((d_{f}(0,y)-\varepsilon(\sup A- \inf A))+(d_f(y,x)-\varepsilon)) \\ & =d_f(0,x)-\varepsilon(1+\sup A- \inf A), \end{split}
\end{equation}
where the second inequality comes from the fact that for points with $\ell_1$ distance $\varepsilon$ we have that the difference of their $d_f$ and $d_{\tilde{f}}$ distance is at most $\varepsilon(\sup A- \inf A)$. 

Using (16), we might finish the proof swiftly. We have seen that it would be sufficient to verify that $B_{\tilde{f}}$ is contained by a small neighborhood of $B_f$. We state that it holds for the neighborhood of $\ell_1$-radius $\frac{\varepsilon(1+\sup A- \inf A)}{\inf A}$. Indeed, consider $x$ so that it is not in this neighborhood. Then for any $z\in\partial B_f$ we have $|z-x|>\frac{\varepsilon(1+\sup A- \inf A)}{\inf A}$.  Now consider any piecewise linear topological path $\Gamma$ from 0 to $x$ with last point $z$ in $\partial B_f$. By (16), as $|z|\geq\frac{1}{\sup A}>\varepsilon$ we obtain
\begin{equation}
\int_{\Gamma}\tilde{f}\geq d_{\tilde{f}}(0,z)+d_{\tilde{f}}(z,x)>d_f(0,z)-\varepsilon(1+\sup A- \inf A)+|z-x|\inf A>1,
\end{equation}
that is $x\notin B_{\tilde{f}}$. Thus we have that $B_{\tilde{f}}$ is in the neighborhood of $B_f$ with radius $\frac{\varepsilon(1+\sup A- \inf A)}{\inf A}$. Consequently, it is in the neighborhood of $K$ with radius $\frac{\varepsilon(1+\sup A- \inf A)}{\inf A}+\varepsilon$. Thus the Hausdorff distance of $K$ and $B_{\tilde{f}}$ can be arbitrarily small, which verifies $K\in\mathcal{W}_A^d$. \end{proof}

Using a construction similar to the one we have just seen, we can prove Theorem 1.3. 

\begin{proof}[Proof of Theorem 1.3] As in the proof of Lemma 3.5, it is sufficient to consider convex polytopes with rational vertices in $\mathcal{K}_A^d$, and we can also assume that they have no boundary point in $\partial D_{\frac{1}{\sup A}}$. Let $\varepsilon>0$ be rational and smaller than $\frac{1}{\sup A}$ and let $x_1,...,x_k,\  \tilde{f}$ be as in the proof of the lemma with $d_H(K,B_{\tilde{f}})<\varepsilon':=\frac{\varepsilon(1+\sup A- \inf A)}{\inf A}+\varepsilon$. This original $\tilde{f}$ had value $\sup A$ everywhere except for certain pieces of the line segments $I_i$, where it equaled $\inf A$. These pieces contained a certain $\lambda_i$ ratio of $|I_i|$. By approximating $\lambda_i$ with rational numbers $(\lambda_{i,j})_{j=1}^{\infty}$ and using them for this ratio, we obtain approximating functions $(\tilde{f_{j}})_{j=1}^{\infty}$ so that $B_{\tilde{f_{j}}}$ converges to $B_{\tilde{f}}$ in the Hausdorff metric as the sum of the lengths of line segments where we modify $\tilde{f_j}\neq\tilde{f}$ can be arbitrarily small. Hence we can assume that $\tilde{f}$ is already defined such that these ratios are rational: each $I_i$ is partitioned into line segments of equal length $s_i$, and each line segment is divided into line segments of rational length $s_{i,1}$ and $s_{i,2}$, and $\tilde{f}$ has value $\inf A$ in the latter ones. We will modify this function in two steps.

We know that for $\varepsilon$ we have that there is a $d_{\tilde{f}}$-optimal path from $\partial D_{\varepsilon}$ to any $x\in I_i\setminus D_{\varepsilon}$, notably the line segment $[{x_i}',x]$ for ${x_i}'=[0,x_i]\cap \partial D_{\varepsilon}$. Moreover, from (i)-(iii) it is simple to see that for small enough $s$ there is a constant $c>0$ such that this line segment is shorter in $d_{\tilde{f}}$ than any other topological path sharing a segment with another $[{x_j}',x_j]$ by at least $c$. We will capitalize on this fact by a bit technical, but necessary argument. Let $\theta>0$ be small enough to be fixed later, and consider the cone $C_i$ with vertex $0$ and base $B_i=\{y: |y|=|x_i|, |y-x_i|\leq\theta\}$, that is the union of all the line segments $[0,y]$ for $\{y: |y|=|x_i|, |y-x_i|\leq\theta\}$. If $\theta$ is small enough, these cones are disjoint. We will define the $g$ modification of $\tilde{f}$ the following way: if $y\in D_{\varepsilon}$, then let $g=\sup A$. If $y\in{C_i \setminus D_{\varepsilon}}$, then let $g(y)=\tilde{f}(x)$ where $x$ is the unique point of $I_i$ with $|x|=|y|$. Apart from these sets we simply let $g=\tilde{f}=\sup A$. By the existence of $c$ we can easily deduce the following claim: for small enough $\theta$ we have that for any $x\in {C_i\cap\partial D_{\varepsilon}}$ there is a $d_{g}$-optimal topological path from $\partial D_{\varepsilon}$ to $x$, and this path is completely contained by $C_i$. Indeed, as there is an $\ell_1$-optimal topological path in $C_i$ and apart from these cones we have $g=\sup A$, we know that if there is a shorter topological path $\Gamma$ in $d_{g}$ from the boundary to some $x\in C_i$ then it must hit another $C_j$. We can assume that $\Gamma$ starts from some $y\in C_j$ with $i\neq j$. However, by choosing $\theta$ small we can have an arbitrarily small bound on the difference of the integrals $\int_{\Gamma}\tilde{f}$ and $\int_{\Gamma}g$. Combining this remark with the existence of $c$ yields our claim, which can be used to give a bound on how large $B_{g}$ can be compared to $B_{\tilde{f}}$ in a manner similar to the concluding step of the proof of Lemma 3.5. After these technical manipulations we obtain using small enough $\theta$ that $d_H(K,B_{g})<2\varepsilon'$. We will use a further modified version $\tilde{g}$ of $g$, which is obtained as follows: focus on a certain $I_i\setminus D_{\varepsilon}$ and $C_i \setminus D_{\varepsilon}$. This line segment is divided into pieces $I_{i,1,1},I_{i,1,2},I_{i,2,1},I_{i,2,2},...,I_{i,l,1},I_{i,l,2}$ with lengths alternating between $s_{i,1}$ and $s_{i,2}$. Denote the corresponding slices of $C_i$ by $C_{i,1,1},C_{i,1,2},C_{i,2,1},C_{i,2,2},...,C_{i,l,1},C_{i,l,2}$. Consider now any of these line segments, for example $[a,b]=I_{i,1,1}$. (For any other line segment we can proceed the same way.) Here $a,b\in\mathbb{Q}^d$, hence for all the $mt$ multiples with $t\in\mathbb{N}$ of a certain $m$ we have $a,b\in\frac{\mathbb{Z}^d}{mt}$. As $C_{i,1,l}$ contains an open cylinder with height $[a,b]$, if we choose a large enough such $mt$ we have an $\ell_1$ optimal path from $a$ to $b$ which is a path in the graph $\frac{\mathbb{Z}^d}{mt}$. Now if we choose $m$ to be a large enough common multiple $M$ of all the (finitely many) $m$s appearing this way, we arrive at a $\frac{\mathbb{Z}^d}{M}$ for which all the $[a,b]$s can be replaced by a path in the graph $\frac{\mathbb{Z}^d}{M}$ the above way. After all we obtain paths $\Gamma_i\subseteq C_i$ for each $i$ which are $\ell_1$-optimal from ${x_i}'$ to $x_i$. Now in each $\Gamma_i$ we will define $\tilde{g}$ to have the same value as $g$, but apart from that we let $\tilde{g}=\sup A$. This way we obtain $\tilde{g}\geq{g}$, hence $B_{\tilde{g}}\subseteq{B_g}$ holds, which yields that the $2\varepsilon'$-neighborhood of $K$ necessarily contains $B_{\tilde{g}}$. However, by definition along the path $\Gamma_i$ the integral of $\tilde{g}$ is the same as the integral of $\tilde{f}$ along $[{x_i}',x_i]$, which can be arbitrarily close to 1. Consequently, we have that an arbitrarily large piece of $\Gamma_i$ is contained by $B_{\tilde{g}}$. Hence as the $\varepsilon'+\theta$-neighborhood of $\bigcup_{i=1}^{k}\Gamma_i$ contains $K$, we have the same for the $\varepsilon'+\theta$-neighborhood of $B_{\tilde{g}}$. Consequently, if $\theta<\varepsilon'$, we have $d_H(K,B_{\tilde{g}})<2\varepsilon'$. In the following we will use this $\tilde{g}$ which is defined based on the parameters $\varepsilon,M$ which are to be fixed later.

Let us return to the statement of the theorem. By separability arguments it suffices to prove that for any polytope $K$ with rational vertices and no boundary points in $D_{\frac{1}{\sup A}}$, in a residual subset of $\Omega$ there exists a suitable sequence $t_n\to\infty$ with $\frac{B(t_n)}{t_n}\to K$. Denote the set of configurations for which it does not hold by $F(K)$. Then by the definition of convergence, $F(K)$ is expressible as a countable union as follows:
\begin{displaymath}
F(K)=\bigcup_{i=1}^{\infty}\bigcup_{l=1}^{\infty}F\left(K,\frac{1}{i},l\right),
\end{displaymath}
where $F\left(K,\delta,\mu_0\right)$ stands for the set of configurations in which for any $\mu>\mu_0$ we have
\begin{displaymath}
d_H\left(K,\frac{\tilde{B}(\mu)}{\mu}\right)>\delta.
\end{displaymath}
Consequently, verifying that $F\left(K,\delta,l\right)$ is nowhere dense for each $\delta,l$ would conclude the proof. Clearly it suffices to prove it for sufficiently small $\delta$ and sufficiently large $l$. Having this purpose in mind fix a cylinder set $U$ in $\Omega$ with nontrivial projections $U_{e_1}, ... U_{e_j}$. We try to find a smaller cylinder set $V=V_t$ such that it is disjoint from $F\left(K,\delta,l\right)$. Now choose $l$ so large that the edges $\frac{e_i}{l}\subseteq D_{\varepsilon}$. Consider the function $\tilde{g}$ defined in the first step of the proof for some $M>l$ and $\varepsilon$ to be fixed later. This function is constant by definition on any edge of $\frac{\mathbb{Z}^d}{Mt}$ for $t\in\mathbb{N}$, denote this value by $\tilde{g}\left(\frac{e}{Mt}\right)$. Now for any edge $e\notin\{e_1,...,e_j\}$, but intersecting $D_{\frac{Mt}{\inf A}}$, we will define $V_{t,e}=(\tilde{g}\left(\frac{e}{Mt}\right)-\varepsilon_e, \tilde{g}\left(\frac{e}{Mt}\right)+\varepsilon_e)\cap A$, where the sum of these $\varepsilon_e$s is at most $\varepsilon\inf A$. Let $V_t$ be defined by these projections and consider any configuration $\omega$ in it. As in the proof of Theorem 3.3, the passage times in this configuration determine a function $f_{t,\omega}$ by rescaling, such that the Hausdorff distance of $\frac{B(Mt)}{Mt}$ and $B_{f_{t,\omega}}$ converges to 0 as $t\to\infty$. (Here we use the fact we noted there that $B_{f_{t,\omega}}$ depends only on the values of $f_{t,\omega}$ in $D_{\frac{1}{\inf A}}$). However, by definition the function $f_{t,\omega}$ is almost the same as $\tilde{g}$ in $D_{\frac{1}{\inf A}}$: the only differences arise due to the existence of the edges $e_1,...,e_j$ and the error term $\varepsilon_e$ for each edge. However, these minor differences cannot imply a significant deviation of the integral on any relevant path in the definition of the sets $B_{\tilde{g}}$ and $B_{f_{t,\omega}}$: as relevant topological paths has $\ell_1$-length at most $\frac{1}{\inf A}$ the error terms may not yield a difference larger than $\varepsilon$ in the integral of $f_{t,\omega}$ and $\tilde{g}$. On the other hand, the edges $\frac{e_1}{l},...,\frac{e_j}{l}$ are all in $D_{\varepsilon}$, which is a set with $\ell_1$-diameter $2\varepsilon$, hence they cannot yield a difference larger than $2\varepsilon(\sup A - \inf A)$. As depending on the choice of $\varepsilon$ all these quantities can be arbitrarily small, we have that the Hausdorff distance of $B_{\tilde{g}}$ and $B_{f_{t,\omega}}$ can be arbitrarily small. Consequently, for small enough $\varepsilon$, and large enough $t$ we have $d_H(B_{\tilde{g}},\frac{B(Mt)}{Mt})<\frac{\delta}{2}$. But for small enough $\varepsilon$ we also have $d_H(B_{\tilde{g}}, K)<\frac{\delta}{2}$. That is, we have by triangle inequality
\begin{displaymath}
d_H\left(K,\frac{\tilde{B}(Mt)}{Mt}\right)<\delta,
\end{displaymath}
for some $Mt>l$. It means that for large enough $t$ we have that $V_t$ is necessarily disjoint from $F\left(K,\delta,l\right)$, which yields that this latter set is nowhere dense. We have seen that it concludes the proof. \end{proof}

\section{Hilbert percolation -- strongly positively dependent case}

In this section, we will restrict our observations to finite dimensional real Hilbert spaces, i.e. $\mathcal{H}=\mathbb{R}^k$.

\begin{defi}
Let $H\subseteq\mathcal{H}$. The convex cone generated by $H$ is the smallest set $\cone(H)\subseteq\mathcal{H}$ which is closed under linear combinations with nonnegative coefficients. The closed convex cone generated by $H$ is the closure of $\cone(H)$ which we denote by $\overline{\cone(H)}$.
\end{defi}

\begin{defi} 
Let $H\subseteq\mathcal{H}$. We say that $H$ is strongly positively dependent if for any $x\in{H}$ we have $-x\in\overline{\cone(H)}$.
\end{defi}

At first sight, one might believe that if $A$ is strongly positively dependent then the passage time between any two points is 0 for a generic configuration. However, a very simple counterexample refutes this idea: if $d=1$ and $A=\{-1,1\}\subseteq{\mathbb{R}}$ for example, then clearly the passage time between integers of different parity is odd, thus cannot be 0. In general we must point out that we do not have paths of any $\ell_1$-length between fixed lattice points, which causes technical inconveniences. We can only construct paths which have length with the same parity as $|x-y|$. Motivated by this remark it is useful to introduce the notation $M(A)$ for the submonoid in $\overline{\cone(A)}$ which contains those linear combinations with nonnegative integer coefficients in which the sum of coefficients is even.

The following theorem displays that in the generic case the model gets trivial if $A$ is strongly positively dependent and we have enough degree of freedom, that is $d\geq{2}$.

\begin{theorem} Let $\mathcal{H}=\mathbb{R}^k$ and assume that $A\subseteq\mathcal{H}$ is bounded and strongly positively dependent, moreover, let $d\geq{2}$. Then in a residual subset of $\Omega$ we have that $\overline{S(x,y)}=\overline{M(A)}$ for any $x,y\in\mathbb{Z}^d$ for which $|x-y|$ is even, and $\overline{S(x,y)}=\overline{M(A)+A}$ for any $x,y\in\mathbb{Z}^d$ for which $|x-y|$ is odd. \end{theorem}

\begin{kov} In the setting of Theorem 4.3, in a residual subset of $\Omega$ we have $T(x,y)=0$ for any $x,y\in\mathbb{Z}^d$ for which $|x-y|$ is even, while $T(x,y)=\inf\{\|b\|_{\mathcal{H}}: b=m+a, m\in M(A), a\in A\}$ for any $x,y\in\mathbb{Z}^d$ for which $|x-y|$ is odd. \end{kov}

\begin{megj}
If $d=1$, it is easy to construct counterexamples to the theorem and to the corollary, due to the fact that for given $x,y$ and $e\in E$ the parity of the times a path $\Gamma$ from $x$ to $y$ crosses $e$ is completely determined. Consequently, if for example $A=\{-1,0,1\}$, while $\tau([0,1])=1$ and $\tau([1,2])=0$, then the passage time of any $\Gamma$ from $0$ to $2$ is odd, thus $T(0,2)\geq 1$ for any configuration, while it should be 0 generically according to Corollary 4.4 as $A$ is strongly positively dependent.
\end{megj}

\begin{proof}[Proof of Theorem 4.3]

Let us consider points with even $\ell_1$-distance, the other case will simply follow from that. The $\overline{S(x,y)}\subseteq\overline{M(A)}$ containment clearly holds as all the passage vectors between points even distance apart are in $M(A)$. Hence  it suffices to verify the other containment which follows from $M(A)\subseteq\overline{S(x,y)}$. For technical purposes choose a countable dense subset $M_0\subseteq M(A)$. Then we can reduce the problem to the question whether $M_0\subseteq\overline{S(x,y)}$ holds. As there are countably many pairs $x,y$ and $M_0$ is also countable, it is enough to show that for any fixed $m\in{M_0}$ and pair $x,y$ we have that $m\in\overline{S(x,y)}$ apart from a nowhere dense set.

In order to verify this claim fix a cylinder set $U\subseteq{\Omega}$ and $n\in\mathbb{N}$. It is clearly sufficient to find a smaller cylinder set $V\subseteq U$ such that for any configuration in $V$ there exists $s\in S(x,y)$ such that $\|m-s\|_{\mathcal{H}}<\frac{1}{n}:=\varepsilon$. Let us denote the set of edges belonging to nontrivial projections of $U$ by $E_U=\{e_1, e_2, ..., e_N\}$. Fix now a large hypercube $K$ centered at the origin which contains all the edges in $E_U$ and also $x$ and $y$. Now we can define a cylinder set $U'$ which has very small projections to the edges in $K$. More precisely, we require these projections to have sufficiently small diameter to guarantee that in $U'$ the sum of passage vectors on these edges is in an open set of diameter $\frac{\varepsilon}{2}$, regardless of which configuration we consider. Choose now a vertex $z$ on $\partial{K}$ and let $\Gamma$ be a path crossing each of its edges precisely once from $x$ to $y$ such that it does not leave $K$ but contains $z$ once. The point $z$ cuts it into two parts $\Gamma_1,\Gamma_2$. By the previous note about the diameters, we have that the passage vector of $\Gamma$ is in an open set of diameter $\frac{\varepsilon}{2}$. Let us denote one of these passage vectors by $\alpha$ for the sake of specificity, all the others that may arise in another configuration have distance at most $\frac{\varepsilon}{2}$ from it. Our aim is to define nontrivial projections on further edges of a cycle $\Gamma'$ starting from $z$ such that the passage vector of $\Gamma^*=\Gamma_1\cup\Gamma'\cup\Gamma_2$ is in a neighborhood of $m$ with radius $\varepsilon$, which would conclude the proof of the first part. As $z$ is a vertex, we will be able to choose $\Gamma'$ so that it does not contain any of its edges twice and $\Gamma'\cap K = \{z\}$.

As $\alpha$ is a passage vector between $x$ and $y$ for some configuration, we clearly have $\alpha\in\cone(A)$. As $A$ is strongly positively dependent, it clearly implies $-\alpha\in\overline{\cone(A)}$. Hence for any $Q\in\mathbb{N}$ there exists $\beta\in\cone(A)$ such that 
\begin{equation}
\|\beta-(-\alpha)\|_{\mathcal{H}}<\frac{\varepsilon}{8Q}=\varepsilon^*,
\end{equation} 
where $Q$ is to be fixed later. By a simple consequence of Carathéodory's theorem about convex hulls we have that
\begin{equation}
\beta=\sum_{i=1}^{k}r_i a_i,
\end{equation}
where each coefficient $r_i>0$, while $a_i\in{A}$. By (19), we can rewrite (18) as
\begin{equation}
\left\|\sum_{i=1}^{k}r_i a_i+\alpha\right\|_{\mathcal{H}}<\varepsilon^*.
\end{equation} 
By the simultaneous version of Dirichlet's approximation theorem we can choose $p_1,...,p_k\geq{0}$ integers and $1\leq{q}\leq{Q}$ such that
\begin{equation}
\left|\frac{p_i}{q}-r_i\right|<\frac{1}{qQ^\frac{1}{k}}.
\end{equation}
Using (21) and the triangle inequality we can rewrite the estimate in (20) as
\begin{equation}
\left\|\sum_{i=1}^{k}\frac{p_i}{q} a_i+\alpha\right\|_{\mathcal{H}}<\varepsilon^*+\frac{1}{qQ^\frac{1}{k}}\sum_{i=1}^{k}\|a_i\|_{\mathcal{H}}.
\end{equation} 
Multiplying by $2q$ and using $q\leq{Q}$ and the definition of $\varepsilon^*$ yields
\begin{equation}
\left\|\sum_{i=1}^{k}2p_i a_i+2q\alpha\right\|_{\mathcal{H}}<\frac{\varepsilon}{4}+\frac{2}{Q^\frac{1}{k}}\sum_{i=1}^{k}\|a_i\|_{\mathcal{H}}<\frac{\varepsilon}{3}
\end{equation}
for well-chosen $Q$, as $A$, and hence $\sum_{i=1}^{k}\|a_i\|_{\mathcal{H}}$ is bounded.

As $m\in{M_0}\subseteq{M(A)}$, by definition it is expressible as an even sum of elements in $A$, that is $m=\sum_{j=1}^{2l}a_{m(j)}$, where each $a_{m(j)}$ is in $A$. Now choose the cycle $\Gamma'$ starting from $z$ such that it does not hit $K$ until eventually returning to $z$, and contains all its edges precisely once. Moreover, $|\Gamma'|=2l+\sum_{i=1}^{k}2p_i+(2q-1)|\Gamma|$. (It is an even number, so it can be done.) On this cycle we can define nontrivial projections the following way: on the first $2l$ edges, we define nontrivial projections centered at each $a_{m(j)}$ respectively with sufficiently small diameter to be precised later. On the next $2p_1,2p_2,...,2p_k$ edges we define nontrivial projections centered at $a_1,a_2,...a_k$ respectively with sufficiently small diameter again. Finally, we think of the last $(2q-1)|\Gamma|$ edges as $2q-1$ consecutive copies of $\Gamma$, that is we define nontrivial projections as sufficiently small open subsets of the projections belonging to the corresponding edges of $\Gamma$. These projections together with the ones in the definition of $U'$ define $V$. If we choose the above neighborhoods small enough, we can guarantee that the passage vector of $\Gamma'$ for any configuration has distance at most $\frac{\varepsilon}{6}$ from $m+\sum_{i=1}^{k}2p_i a_i+(2q-1)\alpha$. As a consequence, by (23) and triangle inequalities we can deduce
\begin{equation}
\|v(\Gamma_1\cup\Gamma'\cup\Gamma_2)-m\|_{\mathcal{H}}<\left\|\sum_{i=1}^{k}2p_i a_i+2q\alpha\right\|_{\mathcal{H}}+\frac{\varepsilon}{6}+\frac{\varepsilon}{2}<\varepsilon
\end{equation}
for any configuration in $V$, which concludes the case when $|x-y|$ is even with the choice $s=v(\Gamma_1\cup\Gamma'\cup\Gamma_2)\in S(x,y)$.

In the other case the previous argument might be copied with one essential change. In this case we want to have passage vectors near vectors of the type $m+a\in M(A)+A$. Given this, upon defining $\Gamma$ and the projections to the edges of $K$, we will proceed the same way as previously, except for this time we will separate an edge $e\in\Gamma\setminus E_U$ and on that we will define the projection of $V$ to be a very small neighborhood of $a$. Apart from this edge, $\Gamma$ uses an even number of edges, thus we can define $\Gamma'$ and $V$ as previously in order to have that $v(\Gamma_1\cup\Gamma'\cup\Gamma_2)-a$ is very close $m$ in $V$. Consequently, $v(\Gamma_1\cup\Gamma'\cup\Gamma_2)$ is very close to $m+a$. The technical details are left to the reader.\end{proof}

It would be nice to say something about how common are the strongly positively dependent sets for example amongst the compact sets equipped with the Hausdorff metric, whose family shall be denoted by $\mathcal{K}^d$. This is the aim of the following proposition, which roughly states that the set of such $A$s is not too small, but not too large either:

\begin{prop}
The set of strongly positively dependent compact sets contains nontrivial open sets in $\mathcal{K}^d$, and so does its complement.
\end{prop}

\begin{proof}

For the complement it is very simple to verify the claim: we can consider the ball of radius $\frac{1}{2}$ centered at the singleton $\{\xi_1\}$. It is good indeed as any set $A$ in this neighborhood exclusively contains vectors with positive first coordinate.

For the set of strongly positively dependent compact sets, our construction relies on the following remark: if the convex hull $\conv(A)$ contains $0$ in its interior, then $A$ is strongly positively dependent. Indeed, in this case for any $a\in A$ we have that $\lambda (-a)\in \conv(A)$ for sufficiently small $\lambda>0$. Consequently, $\lambda (-a)$ can be written as a finite linear combination of elements of $A$ with positive coefficients, which yields that $-a\in \cone(A)$, as stated.

Now consider $A=\{\pm\xi_1,\pm\xi_2,...,\pm\xi_k\}$ in $\mathcal{K}^d$. Then $\conv(A)$ contains $0$ in its interior, as it contains $D_1$, the unit ball centered at the origin in the $\ell_1$-metric. In other words, the distance of $0$ from $\partial(\conv(A))$ is 1, and $0$ is contained by $\conv(A)$. Now consider a small neighborhood $G_A$ of $A$ in $\mathcal{K}^d$ and an element $K$ of it. We would like to show that for a sufficiently small neighborhood we have that $0$ is in the interior of $\conv(K)$. We know that $K$ contains at least one point very close to each $\pm\xi_i$, and if $G_A$ is small enough, these points must be distinct. If we replace $K$ by a subset of it formed by $2k$ such points, we shrink $\conv(K)$, hence it would be sufficient to verify our claim for the convex hulls of such finite sets. But such a convex hull is a polytope, which itself and whose boundary is a continuous function of the vertices. Consequently, as we know that the distance of $0$ from $\partial(\conv(A))$ is 1, and $0$ is in $A$, we have that the distance of $0$ from $\partial(\conv(K))$ is also positive and $0$ is in $K$ if $K$ is chosen from a sufficiently small neighborhood $G_A$. Thus the set of strongly positively dependent compact sets contains nontrivial open sets in $\mathcal{K}^d$, indeed. \end{proof}

\section{Optimal paths and geodesics in Hilbert percolation}

In the original setup we called a path geodesic if its passage time equals the passage time between its endpoints, which was appropriate in the sense that subpaths of geodesics were also of minimal length. However, as even very simple examples might display, it is not the case anymore: for instance, let $A=\{-1,0,1\}$ and $d=2$, and consider the configuration in which we have two neighboring edges with passage vectors -1 and 1 respectively, while the passage vector of all other edges is 0. In this case, the passage time of the path of these two edges is 0, hence it is optimal, while its subpaths of length 1 have passage time 1. However, between any two points the passage time is 0, hence these subpaths are not optimal. It motivates a separation of definitions:

\begin{defi}
A path in $\mathbb{Z}^d$ is an optimal path, if its passage time equals the passage time between its endpoints. Moreover, a path is a geodesic, if all of its subpaths are optimal.
\end{defi}

If we would like to have a somewhat tame geometry, it is natural to expect from the model that optimal paths are not self intersecting, and in general longer paths have higher passage time. The following proposition gives a necessary and sufficient condition guaranteeing this property. (We denote by $(a,b)$ the inner product of $a,b\in\mathcal{H}$).

\begin{lemma}
We have $\tau(\Gamma_1)\leq\tau(\Gamma_2)$ for each paths $\Gamma_1\subseteq \Gamma_2$ and each configuration if and only if for any $a,b\in A$ we have $(a,b)\geq{0}$. In this case, we call $A$ positive.
\end{lemma}

\begin{proof}
First assume the existence of $a,b\in A$ with $(a,b)<0$. Choose $n\in\mathbb{N}$ so that $2n(a,b)+\|b\|_{\mathcal{H}}^2<0$. Now consider a configuration in which there are $n+1$ consecutive edges so that the passage time of the first $n$ is $a$, while the last one has passage time $b$. Let the first $n$ edges form $\Gamma_1$, and let the union of all these edges be $\Gamma_2$. Then the square of the passage time of $\Gamma_1$ is
\begin{displaymath}
\tau(\Gamma_1)^2=\|na\|_{\mathcal{H}}^2=n^2\|a\|_\mathcal{H}^2,
\end{displaymath}
which implies by the choice of $n$
\begin{displaymath}
\tau(\Gamma_2)^2=\|na+b\|_{\mathcal{H}}^2=n^2\|a\|_\mathcal{H}^2+2n(a,b)+\|b\|_\mathcal{H}^2<n^2\|a\|_\mathcal{H}^2=\tau(\Gamma_1)^2.
\end{displaymath}
This denies $\tau(\Gamma_1)\leq\tau(\Gamma_2)$, hence we proved one of the directions.

For the other direction assume that for any $a,b\in A$ we have $(a,b)\geq{0}$, and $\Gamma_1\subseteq \Gamma_2$. As we can add edges one by one, it suffices to prove the claim for $\Gamma_2=\Gamma_1\cup e$ for an edge $e$. Now the square of the passage time of $\Gamma_1$ in any configuration is
\begin{displaymath}
\tau(\Gamma_1)^2=\left\|\sum_{i=1}^{|\Gamma_1|}a_i\right\|_{\mathcal{H}}^2.
\end{displaymath}
for some $a_i\in A$, while the square of the passage time of $\Gamma_2$ is
\begin{displaymath}
\tau(\Gamma_2)^2=\left\|a^*+\sum_{i=1}^{|\Gamma_1|}a_i\right\|_{\mathcal{H}}^2=\|a^*\|_{\mathcal{H}}^2+2\sum_{i=1}^{|\Gamma_1|}(a,a_i)+\tau(\Gamma_1)^2\geq \tau(\Gamma_1)^2,
\end{displaymath}
where the last inequality holds by assumption. It concludes the proof. \end{proof}

Now if $A$ is not positive, then in a small neighborhood of the configuration constructed in the proof there exist paths $\Gamma_1\subseteq \Gamma_2$ with $\tau(\Gamma_1)>\tau(\Gamma_2)$. Thus if $\Omega$ is a Baire space, then we have $\tau(\Gamma_1)\leq\tau(\Gamma_2)$ in the generic case for each paths $\Gamma_1\subseteq \Gamma_2$ if and only if $A$ is positive. Thus in the following we will restrict ourselves to positive $A$s. A natural question is when we have that all the optimal paths are geodesics. If $d=1$, we clearly have this property as optimal paths are not self-intersecting, and if $d=1$ there is a unique path with no self-intersections between any two points. Things get interesting when $d\geq{2}$, however, we must realize that such $A$s are terrifyingly rare:

\begin{lemma}
Assume that $A$ is positive and $d\geq{2}$. We have that all the optimal paths are geodesics in all the configurations if and only if $A$ is contained by a ray, that is for any $a,b\in A$ we have $a=\lambda b$ or $b=\lambda a$ for some $\lambda\geq{0}$. Moreover, if this condition does not hold, there are configurations in which there are $x,y\in\mathbb{Z}^d$ such that there is no geodesic between $x$ and $y$ at all.
\end{lemma}

\begin{proof}
If $A$ is contained by a ray we obviously have this property as the passage time of a path is simply the sum of the passage times of the edges. On the other hand, assume that $A$ contains $a,b$ so that they are not contained by the same ray. We can clearly assume $\|a\|_{\mathcal{H}}\geq\|b\|_\mathcal{H}$. Moreover, by the assumption we have $(a,b)\leq \mu \|a\|_{\mathcal{H}}\|b\|_{\mathcal{H}}$ for some $0<\mu<1$, and $a,b\neq 0$. The proof from this point is a construction in which we have an optimal path which is not a geodesic.

First consider the case $\|a\|_{\mathcal{H}}=\|b\|_{\mathcal{H}}$. For this case we can give a very simple construction, see Figure \ref{abra1}. Notably, amongst the paths from $X$ to $Y$ there is a unique one with optimal $\ell_1$-length and passage vector $a+a+b+b$, notably one of the paths through $Z$. Consequently, this path is the unique optimal path from $X$ to $Y$. However, it cannot be a geodesic, as there is a path from $X$ to $Z$ with passage vector $b+a$, hence it has smaller passage time than the path with passage vector $a+a$. Thus there are no geodesics from $X$ to $Y$ at all.

\begin{figure}[h!]
  \includegraphics[width=200pt]{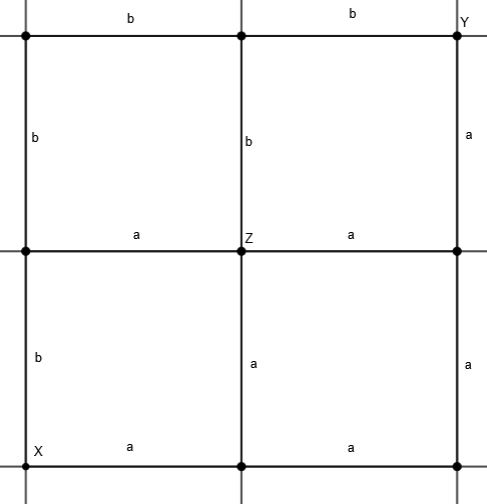}
  \caption{The case when $\|a\|_{\mathcal{H}}=\|b\|_{\mathcal{H}}$.}
  \label{abra1}
\end{figure}

Consider the other case, when there is a strict inequality between the norms of $a$ and $b$, that is $\|a\|_{\mathcal{H}}>\|b\|_{\mathcal{H}}$. For the norm of $a$ and $b$ we can define $n_a<n_b$ such that both of them are even and the inequalities
\begin{displaymath}
n_a\|a\|_{\mathcal{H}}>n_b\|b\|_{\mathcal{H}}, \text{ } n_a\|a\|_{\mathcal{H}}\leq n_b(\|b\|_{\mathcal{H}}+\varepsilon)
\end{displaymath}
simultaneously hold for some $\varepsilon>0$ to be fixed later. We consider the following configuration: on the line segment $\Gamma_1=[0,n_a\xi_1]$ let all the passage vectors be $a$, while in the line segment $\Gamma_3=[n_a\xi_1,n_a\xi_1+3n_b\xi_2]$ let the passage vectors be $b$. Moreover, on the line segments 
$$[0,-\frac{n_b-n_a}{2}\xi_2],\text{ } [-\frac{n_b-n_a}{2}\xi_2,n_a\xi_1-\frac{n_b-n_a}{2}\xi_2], \text{ } [n_a\xi_1-\frac{n_b-n_a}{2}\xi_2, n_a\xi_1]$$
let all the passage vectors be $b$. (The union of these line segments will be denoted by $\Gamma_2$.) Concerning the remaining edges we separate two cases. (See Figure \ref{orthogonalpr} for the discussion belonging to one of them.) Notably, we know that if we consider the line $L=\{ta+((3n_b+n_a)-t)b \text{ : } t\in\mathbb{R}\}\subseteq\mathcal{H}$, then there is a unique $t=t_0\in\mathbb{R}$, for which the norm of $p_{t_0}=t_0a+((3n_b+n_a)-t_0)b$ is minimal, that is the orthogonal projection of the origin to $L$.
\begin{figure}[h!]
  \includegraphics[width=350pt]{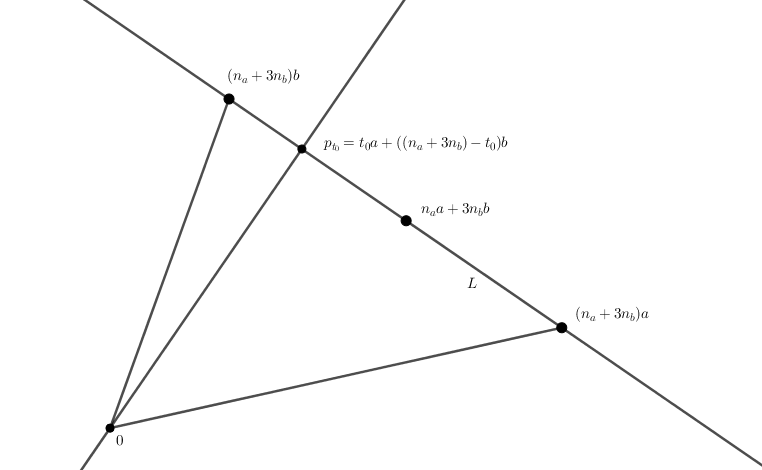}
  \caption{The definition of $L$, example for the first case.}
  \label{orthogonalpr}
\end{figure}
If we increase or decrease $t$ gradually from $t_0$, then due to the orthogonality of $[0,p_{t_0}]$ and $L$, the norm of $p_t$ is strictly increasing in both directions. From this observation we infer that if we consider $t=n_a$, then if we move $t$ towards one of the directions, the norm will strictly increase. Now if this increase shows up in the direction of $(3n_b+n_a)a$, then we place the passage vector $a$ to all the remaining edges. On the other hand, if this increase appears in the direction of $(3n_b+n_a)b$, then we place the passage vector $b$ to all the remaining edges.

Now let us calculate the square of the passage times of $\Gamma_1$ and $\Gamma_2$:
\begin{displaymath}
\tau(\Gamma_1)^2=n_a^2\|a\|_{\mathcal{H}}^2, \text{ } \tau(\Gamma_2)^2=n_b^2\|b\|_{\mathcal{H}}^2.
\end{displaymath}
Consequently, we have $\tau(\Gamma_1)>\tau(\Gamma_2)$ by the choice of $n_a,n_b$, hence $\Gamma_1$ cannot be optimal. Thus if we have that $\Gamma_1\cup \Gamma_3$ is the only optimal path from the origin to $n_a\xi_1+3n_b\xi_2$, that concludes the proof, as one of its subpaths is not optimal, consequently, it cannot be a geodesic. In the following we will examine whether there are other optimal paths between these points and can they be geodesics.

First consider the case in which we placed the passage vector $b$ to all the remaining edges, that is the norm of $n_a a+3n_b b$ is smaller than the norm of $ta+((3n_b+n_a)-t)b$ for any $t<n_a$. In this case any path $\Gamma \neq \Gamma_1\cup\Gamma_3$ uses at least $n_a+3n_b$ edges, and at least $3n_b+1$ of them has passage vector $b$. Thus by the assumption of this case concerning the norms, we certainly have that the norm of the passage vector of $\Gamma$ is larger than the norm of $n_a a+3n_b b=v(\Gamma_1\cup\Gamma_3)$. (At this point we also use the positivity of $A$: if $\Gamma$ uses more edges, than $n_a+3n_b$, then we reduce the passage time by forgetting about edges with passage vector $a$.) Thus this case is concluded, $\Gamma_1 \cup \Gamma_3$ is the only optimal path from the origin to $n_a\xi_1+3n_b\xi_2$.

Consider now the other case, that is the norm of $n_a a+3n_b b$ is smaller than the norm of $ta+((3n_b+n_a)-t)b$ for any $t>n_a$. Assume that there is another path $\Gamma\neq \Gamma_1\cup \Gamma_3$ from the origin to $n_a\xi_1+3n_b\xi_2$ which is optimal. It must hit the line segment $[-\frac{n_b-n_a}{2}\xi_2,n_a\xi_1-\frac{n_b-n_a}{2}\xi_2]$: indeed, otherwise $\Gamma$ would contain at least $n_a$ edges with passage vector $a$. Now if $|\Gamma|>n_a+3n_b$, it immediately implies that the passage time of $\Gamma$ exceeds the passage time of $\Gamma_1\cup \Gamma_3$, a contradiction. On the other hand, if $|\Gamma|=n_a+3n_b$, the number of edges with passage vector $a$ surely exceeds $n_a$. Hence by the starting assumption of this case concerning the norms, the norm of the passage vector of $\Gamma$ is larger than the norm of $n_a a+3n_b b=v(\Gamma_1\cup\Gamma_3)$, a contradiction. Hence $\Gamma$ hits the line segment $[-\frac{n_b-n_a}{2}\xi_2,n_a\xi_1-\frac{n_b-n_a}{2}\xi_2]$ at some point $x$ indeed. Denote its first part from the origin to $x$ by $\Gamma'$ and its second part from $x$ to $n_a\xi_1+3n_b\xi_2$ by $\Gamma''$. Assume that $\Gamma'$ has an edge which is not contained by $\Gamma_2$, and hence its passage vector is $a$. In this case $\Gamma'$ has an entire subpath $\Gamma_0'$ connecting points of $\Gamma'$, and with edges of passage vector $a$. However, $\Gamma_0'$ can be replaced with an $\ell_1$-optimal path such that all of its edges have passage vector $b$. Hence its passage time is lower than $\tau(\Gamma_0')$, yielding that $\Gamma_0'$ cannot be optimal. Consequently, $\Gamma$ has a subpath which is not optimal, while $\Gamma$ is optimal, thus $\Gamma$ cannot be a geodesic. We can proceed similarly if $\Gamma''$ has an edge which is not contained by $\Gamma_2\cup\Gamma_3$. Thus the only case remaining is that $\Gamma$ contains all the edges of $\Gamma_2\cup \Gamma_3$, and hence $\tau(\Gamma)$ is at least $\tau(\Gamma_2\cup \Gamma_3)$. However, in this case the square of its passage time equals $(4n_b)^2\|b\|_{\mathcal{H}}^2$, while the square of passage time of $\Gamma_1\cup \Gamma_3$ can be estimated by
\begin{equation}
\begin{split} \tau(\Gamma_1\cup \Gamma_3)^2 &=n_a^2\|a\|_{\mathcal{H}}^2+6n_a n_b(a,b)+9n_b^2\|b\|_{\mathcal{H}}^2 \\ & \leq n_b^2(\|b\|_{\mathcal{H}}+\varepsilon)^2+6\mu n_b^2\|b\|_{\mathcal{H}}(\|b\|_{\mathcal{H}}+\varepsilon)+9n_b^2\|b\|_{\mathcal{H}}^2. \end{split}
\end{equation}
Thus by $\tau(\Gamma_1\cup \Gamma_3)^2>\tau(\Gamma_2\cup \Gamma_3)^2=(4n_b)^2\|b\|_{\mathcal{H}}^2$ we conclude the following inequality with $\varepsilon=\frac{\|b\|_{\mathcal{H}}}{M}$:
\begin{displaymath}
6n_b^2\|b\|_{\mathcal{H}}^2<6\mu n_b^2\|b\|_{\mathcal{H}}^2+(2+6\mu) n_b^2\|b\|_{\mathcal{H}}\varepsilon+n_b^2\varepsilon^2=\left(6\mu+\frac{(2+6\mu)}{M}+\frac{1}{M^2}\right)n_b^2\|b\|_{\mathcal{H}}^2.
\end{displaymath}
However, for sufficiently large $M$ this inequality cannot hold. This is a contradiction, $\Gamma$ cannot be optimal in this case. Hence we found that there can be no optimal path from 0 to $n_a\xi_1+3n_b\xi_2$ which is a geodesic. It concludes the proof. \end{proof}

\begin{proof}[Proof of Theorem 1.4] From the remark following Lemma 5.2 we know that $A$ is positive by the assumption of the theorem. Moreover, from the proof of Lemma 5.3 it follows very quickly that if $A$ is not contained by a ray then there are nontrivial open sets in which there are no geodesics between certain points, as we can consider sufficiently small neighborhoods of the configurations constructed there. However, as $\Omega$ is a Baire space, nontrivial open sets cannot be disjoint from a residual subset of $\Omega$. Thus $A$ is contained by a ray in fact. It concludes the proof of Theorem 1.4. \end{proof}

\section{Geodesic rays in Hilbert percolation}

\begin{proof}[Proof of Theorem \ref{hilbertrays}]

The case when $A$ is contained by a ray is already covered by Theorem 1.1. Hence we can assume that $A$ is not linearly isomorphic to a subset of the nonnegative reals. Still, the proof will be quite similar to the proof of Theorem 1.1, hence we will focus on the differences this time.

Let $a,b\in A$ so that we cannot get one from another by multiplying with a nonnegative scalar and $\|a\|_{\mathcal{H}}\leq\|b\|_{\mathcal{H}}$. Consequently, we might choose $0<\mu<1$ satisfying $(a,b)\leq \mu\|b\|_{\mathcal{H}}^2$. First we will prove that the origin is the starting point of distinct geodesics in a meager subset only. Let $U, U', K_1, K_2, E_U, E^*$ be as in the proof of Theorem 1.1 defined in terms of the parameters $p,q,q',r$ to be fixed later. We will define $V\subseteq{U'}$ as a cylinder set which has nontrivial projections to the edges in $E_U\cup E^*$. The underlying concept is the same as in that proof: for the configurations in $V$ we would like to have essentially one (and the same) geodesic from the boundary $\partial K_1$ to the boundary $\partial K_2$, notably the line segment connecting $p\xi_1$ and $q\xi_1$. By this we mean that for any lattice points $x_1\in \partial K_1$ and $x_2 \in\partial K_2$, a geodesic $\Gamma$ from $x_1$ to $x_2$ eventually arrives in $p\xi_1$, and then it goes along the line segment $[p\xi_1,q\xi_1]$. It would be sufficient exactly as it was earlier.

We will obtain this property following the same strategy: our purpose is to define $V$ so that sufficiently many paths in $\partial K_1 \cup [p\xi_1,q\xi_1] \cup \partial K_2$ are cheap while other paths in $K_2 \setminus K_1$ are expensive for configurations in $V$. To this end, we use a slightly more complicated construction this time, in which we will guarantee paths to be cheap by having edges with passage vectors $a$ and $b$ alternatingly. To obtain this, consider the connected subgraph $G$ of $\mathbb{Z}^d$ in $\partial K_1 \cup [p\xi_1,q\xi_1] \cup \partial K_2$ and its vertex $p\xi_1$. Let $G_0$ be the breadth-first search tree rooted at vertex $p\xi_1$. By definition, in $G_0$ the unique path from $p\xi_1$ to any other vertex $x\in G$ is optimal in the graph distance of $G$, which is equivalent to being optimal in $\ell_1$ inside $\partial K_1 \cup [p\xi_1,q\xi_1] \cup \partial K_2$. Moreover, as the points of $[p\xi_1,q\xi_1]$ are cut vertices of $G$, that is the deletion of any of them cuts $G$ into two distinct connected components, it obviously implies that for any lattice points $x\in \partial K_1\cup [p\xi_1,q\xi_1)$ and $y\in (p\xi_1,q\xi_1]\cup \partial K_2$, the unique path from $x$ to $y$ contained by $G_0$ is optimal in $G$. Now we can define projections to edges of $G_0$ such that they equal small neighborhoods of $a$ and $b$ and for any such $G$-optimal path these projections appear alternatingly. Indeed, for branches rooted at $p\xi_1$ and proceeding towards $\partial K_2$ let us define the projection to the first edge to be a neighborhood of $b$, and the later ones to be alternatingly neighborhoods of $a$ and $b$. On the other hand, for branches rooted at $p\xi_1$ and contained by $\partial K_1$, let us define the projection to the first edge to be a neighborhood of $a$, and the later ones to be alternatingly neighborhoods of $b$ and $a$. It yields indeed that for any lattice points $x\in \partial K_1\cup [p\xi_1,q\xi_1)$ and $y\in (p\xi_1,q\xi_1]\cup \partial K_2$, on the $G$-optimal path from $x$ to $y$ contained by $G_0$ the projections equal to neighborhoods of $a$ and $b$ alternatingly. On any other edge inside $K_2 \setminus \inte K_1$ let the projection be a small neighborhood of $b$. At this point, we think of all of these projections being the singletons $\{a\}$ and $\{b\}$ respectively, which are not necessarily open in $A$, but might be fattened suitably later. Then the proof relies on the fact that for any $G$-optimal path $\Gamma'$ of the above type with $|\Gamma'|\geq{2}$ we have
\begin{equation}
\begin{split} \tau(\Gamma')^2 & \leq\left\|\frac{|\Gamma'|+1}{2}b+\frac{|\Gamma'|-1}{2}a\right\|_{\mathcal{H}}^2 \\  & \leq \left(\frac{|\Gamma'|+1}{2}\right)^2\|b\|_{\mathcal{H}}^2+\left(\frac{|\Gamma'|-1}{2}\right)^2\|a\|_{\mathcal{H}}^2 + \left(\frac{(|\Gamma'|+1)(|\Gamma'|-1)}{2}\right)(a,b) \\ & \leq \frac{|\Gamma'|^2+1}{2} \|b\|_{\mathcal{H}}^2+ \frac{|\Gamma'|^2-1}{2}(a,b) \leq \frac{3}{4}|\Gamma'|^2\|b\|_{\mathcal{H}}^2+\frac{1}{4}|\Gamma'|^2(a,b) \\ & \leq|\Gamma'|^2\left(\frac{3}{4}+\frac{1}{4}\mu\right)\|b\|_{\mathcal{H}}^2.\end{split}
\label{squareestimate}
\end{equation}
Consequently, for $\lambda=\sqrt{\frac{3}{4}+\frac{1}{4}\mu}<1$ we have
\begin{equation}
\tau(\Gamma')\leq \lambda|\Gamma'|\|b\|_{\mathcal{H}}.
\label{estimate}
\end{equation}
Now proceeding towards a contradiction assume that there is a geodesic $\Gamma_0$ from $\partial K_1$ to $\partial K_2$ which does not contain the line segment $[p\xi_1,q\xi_1]$. Similarly to arguments in the proof of Theorem 1.1, by passing to a subpath we can infer the existence of a geodesic $\Gamma$ from $x_1\in\partial K_1\cup [p\xi_1,q\xi_1)$  to $x_2\in(p\xi_1,q\xi_1]\cup \partial K_2$ which uses only edges not contained by $\partial K_1 \cup [p\xi_1,q\xi_1] \cup \partial K_2$. If $|x_1-x_2|=1$, it is clearly impossible, thus we may assume $|x_1-x_2|>{1}$. Our aim is to show a cheaper path $\Gamma'$ contained by $\partial K_1 \cup [p\xi_1,q\xi_1] \cup \partial K_2$ as that would yield a contradiction. To this end, we construct $\Gamma'$ by a similar method as in the proof of Theorem 1.1: in this case, let it be the unique path from $x_1$ to $x_2$ in $G_0$, which is optimal in $G$. For this path, we can use (\ref{estimate}). Indeed, by this bound, if we want to show $\tau(\Gamma')<\tau(\Gamma)$ to obtain a contradiction, we can do a little trick and replace all the passage vectors of edges in $\partial K_1 \cup [p\xi_1,q\xi_1] \cup \partial K_2$ by $\lambda b$. As $|\Gamma'|\geq{2}$ necessarily, it does not decrease the passage time of $\Gamma'$, and the passage time of $\Gamma$ does not change at all. Consequently, all the passage vectors in $K_2\setminus \inte K_1$ are multiples of $b$, and the ones in $\partial K_1 \cup [p\xi_1,q\xi_1] \cup \partial K_2$ are cheaper than the others. Hence the situation we face is linearly isomorphic to the one in the proof of Theorem 1.1. Thus for a suitable choice of the parameters $p,q,q',r$ we get a contradiction for this specific configuration, $\tau(\Gamma')<\tau(\Gamma)$, as $\Gamma'$ is a path which is optimal in $\ell_1$ amongst the paths contained by $\partial K_1 \cup [p\xi_1,q\xi_1] \cup \partial K_2$, while $\Gamma$ does not have any edge in this set, which is sufficient by the note following that proof. Finally, by taking small enough neighborhoods to be the projections instead of the singletons this inequality will not fail as there are only finitely many pairs $\Gamma,\Gamma'$ to consider. It concludes the proof of the claim that the origin is the starting point of distinct geodesics only in a meager subset. The statement of the theorem is obtained from this claim the same way as in the proof of Theorem 1.1. \end{proof}

To conclude this section, we provide an example for a closed set $A$ so that there are no geodesic rays at all generically. (As $\Omega$ is a Baire space in this case, it means indeed that in a large subset of $\Omega$ there are no geodesic rays.) Let $d=2$, and let $A=\{a,b, 2b\}=\{(1,0),(0,1), (0,2)\}\subseteq{\mathcal{H}}=\mathbb{R}^2$. It suffices to prove that there is no geodesic ray starting from the origin generically. Fix a cylinder set $U$. Let us use a construction similar to the one in the proof of Theorem 1.1: fix a cylinder set $V$ such that there exists $K_1,K_2$ as earlier so that any geodesic ray starting from the origin eventually reaches $q\xi_1$ and do not enter $\inte K_2$ again. More explicitly, in $V$ the passage vectors of edges in $\partial K_1 \cup [p\xi_1,q\xi_1] \cup \partial K_2$ equal $b$ while of other edges in $K_2\setminus K_1$ they equal $2b$. Now it suffices to fix a finite number of further passage vectors so that in such configurations there is no geodesic from $q\xi_1$ not entering $\inte K_2$: that would mean that there is no geodesic ray starting from the origin. We do so by fixing some passage vectors in $q\xi_1+D_7$ to be $a$ or $b$ as shown in Figure \ref{nogeodesics} ($D_7$ denotes the closed ball with $\ell_1$ radius 7 centered at the origin): in the positive half-plane, the red, decorated edges correspond to passage vector $a$, while the blue, simple ones to passage vector $b$. The passage vectors in the negative half-plane are obtained by reflection to the horizontal axis.

\begin{figure}[h!]
  \includegraphics[width=250pt]{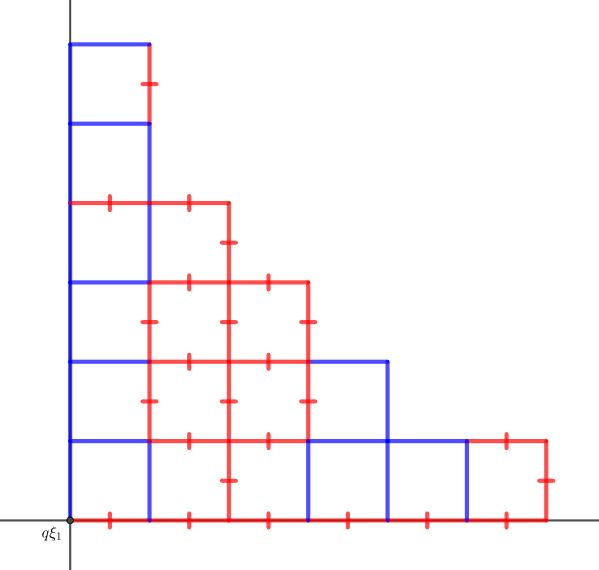}
  \caption{Red decorated edges have passage vector $a$, while blue simple ones have passage vector $b$.}
  \label{nogeodesics}
\end{figure}

Now we claim that there is no geodesic from $q\xi_1$ to points with $q\xi_1+x$ where $|x|=6$ and has nonnegative first coordinate. As we gained the passage vectors in the negative halfplane by reflection, it suffices to prove this claim for $x$ with nonnegative second coordinate, too. Now it is easy to check that if $x=(0,6)$ or $x=(6,0)$ then there are paths with passage vector $4a+4b$ from $q\xi_1$ to $q\xi_1+x$. Meanwhile all the paths with minimal $\ell_1$-length 6 has passage vector $6a$ or $6b$. But $\|4a+4b\|=4\sqrt{2}<6=\|6a\|=\|6b\|$. Consequently, the paths with passage vector $4a+4b$ are the optimal ones, while none of them is a geodesic as their first 2 or first 4 steps cannot be optimal. Thus there is no geodesic to such points. On the other hand, if both of the coordinates of $x$ are positive, then there are optimal paths with passage vector $3a+3b$ from $q\xi_1$ to $q\xi_1+x$. However, their first 2 or first 4 steps are not optimal as they have not been before. Thus there are no geodesics to such points either. Consequently, there is no geodesic from $q\xi_1$ to points with $q\xi_1+x\notin \inte K_2$ where $|x|=6$. Thus the origin cannot be the starting point of a geodesic ray as any such geodesic ray eventually reaches $q\xi_1$ and should continue for an infinite number of steps as a geodesic outside of $\inte K_2$, which is impossible by the previous observation. It yields that there are no geodesic rays at all generically as we claimed.

\section{Concluding remarks and open problems}

In the above sections we answered several questions, but in the meantime new ones arose. In the following, we list a few of them.

In the third section we made some progress concerning the asymptotic behaviour of $\frac{B(t)}{t}$, and we gave a necessary condition concerning any limit set $K=\lim_{n\to\infty}\frac{B(t_n)}{t_n}$, more explicitly we proved that such a set must be in $\overline{\mathcal{W}_A^d}$. The point of view provided by this recognition helped us in proving that all the convex sets of $\mathcal{K}_A^d$ arise as limit sets. However, on one hand, we could not prove that this condition is sufficient, and on the other hand, the definition of the family $\mathcal{W}_A^d$ is quite unhandy and does not really extend our understanding about the limit sets in its own right as it is somewhat just a continuous and more general rephrasing of the definition of a limit set.

\begin{quest} Is it true that all the sets in $\mathcal{W}_A^d$ arise as limit sets? If not, find what is the good family to consider, if yes, try to define it more conveniently. \end{quest}

In the fourth section we understood the nature of passage times in the strongly positively dependent case of the Hilbert first passage percolation in the finite dimensional case for bounded $A$. Our argument relied on Dirichlet's approximation theorem, which has variants for infinite dimensional spaces, too, but they are insufficient for our purposes (see \cite{FSU}). It would be nice to understand the behaviour of the percolation in these cases as well. Moreover, we also capitalized on the boundedness of $A$, hence it should also be examined what happens if this condition is dropped.

\begin{quest} What can be said about the Hilbert percolation in the strongly positively dependent case if the space is infinite dimensional or $A$ is not bounded?\end{quest}

In the sixth section we provided an example for a set $A$ such that there are no geodesic rays at all generically. The proof heavily relied on the simple structure of $A$: roughly we used the fact that there is a configuration in which $q\xi_1$ is not a starting point of geodesics longer than 5 edges. Now it is reasonable to ask if it is always the case when $A$ is not linearly isomorphic to a subset of the nonnegative reals.

\begin{quest} Is it true that there exists a geodesic ray generically if and only if $A$ is linearly isomorphic to a subset of $[0,+\infty)$?\end{quest}

\subsection*{Acknowledgements} 

I am highly grateful to Zolt\'an Buczolich for the time he spent with proofreading the paper, it helped a lot in improving the quality of the exposition. Moreover, I am thankful to Korn\'elia H\'era for the interesting discussion about the earlier results during which she asked the question leading to the second part of this work.


\begin{thebibliography}{99}

\bibitem{AD} {\sc A. Auffinger, M. Damron}, {\it Differentiability at the edge of the limit
shape and related results in first passage percolation}, Probability Theory and Related Fields, {\bf156} (2013), 193-227.

\bibitem{FPP} {\sc A. Auffinger, M. Damron, J. Hanson}, {\it 50 years of first passage percolation}, University Lecture Series, {\bf68} (2016).

\bibitem{B} {\sc M. Bjorklund}, {\it The asymptotic shape theorem for generalized first passage
percolation}, Annals of Probab., \textbf{38} (2010), 632–660.

\bibitem{Bo} {\sc D. Boivin}, {\it First passage percolation: the stationary case}, Probab. Theory
Relat. Fields, \textbf{86} (1990), 491–499.

\bibitem{CD} {\sc J. T. Cox, R. Durrett}, {\it Some limit theorems for percolation with necessary
and sufficient conditions}, Annals of Probab., \textbf{9} (1981), 583-603.

\bibitem{FSU} {\sc L. Fishman, D. S. Simmons, and M. Urbánski}, {\it Diophantine approximation in Banach spaces}, J. Théor. Nombres Bordeaux, {\bf26} no.2. (2014), 363–384.

\bibitem{HW} {\sc J. Hammersley, D. Welsh}, {\it First-passage percolation, subadditive processes,
stochastic networks, and generalized renewal theory}, Proc. Internat. Res.
Semin., Statist. Lab., Univ. California, Berkeley, Calif., Springer-Verlag,
New York, (1965), 61–110.

\bibitem{K} {\sc H. Kesten},  {\it \'Ecole d'\'Et\'e de Probabilit\'es de Saint Flour XIV}, Lecture Notes in Mathematics, {\bf1180} (1986), 125-264.

\bibitem{Ku} {\sc K. Kuratowski}, {\it Topologie}, Vol. 1, 4th ed., PWN, Warsaw, 1958; English transl., Academic Press, New York; PWN, Warsaw, (1966).

\bibitem{M} {\sc B. Maga}, {\it Baire categorical aspects of first passage percolation}, Acta Mathematica Hungarica, \textbf{156} (1) (2018), 145-171.

\bibitem{WR} {\sc J. C. Wierman, W. Reh}, {\it On conjectures in first passage percolation
theory}, Annals of Probab., {\bf6} (1978), 388-397.




\end{thebibliography}
\end{document}